\documentclass[a4paper,12pt]{amsart}
\usepackage{amsmath,amssymb,amsfonts,amsthm}

\usepackage[numeric]{amsrefs}
\usepackage{bm}

\usepackage{graphicx}
\usepackage{ascmac}
\usepackage[all]{xy}
\usepackage{amsthm}
\usepackage[top=25mm,bottom=25mm,left=25mm,right=25mm]{geometry}
\usepackage{tikz}
\usetikzlibrary{cd,arrows,matrix,backgrounds,shapes,positioning,calc,decorations.markings,decorations.pathmorphing,decorations.pathreplacing}
\tikzset{blackv/.style={circle,fill=black,inner sep=3pt,outer sep=3pt},
         whitev/.style={circle,fill=white,draw=black,inner sep=3pt,outer sep=3pt},
         blabel/.style={circle,draw=black,inner sep=1.5pt,outer sep=0pt},
         redv/.style={circle,fill=red,inner sep=3pt,outer sep=3pt},
         block/.style={draw,rectangle split,rectangle split horizontal,rectangle split parts=#1},
         symbol/.style={
           draw=none,
           every to/.append style={
             edge node={node [sloped, allow upside down, auto=false]{$#1$}}}}
}
\usepackage{pxpgfmark}
\usepackage{multirow}
\usetikzlibrary{intersections,arrows,calc,decorations.markings}
\newtheorem{theorem}{Theorem}[section]

\newtheorem{proposition}[theorem]{Proposition}

\newtheorem{conjecture}[theorem]{Conjecture}
\newtheorem{corollary}[theorem]{Corollary}
\newtheorem{lemma}[theorem]{Lemma}

\theoremstyle{definition}
\newtheorem{remark}[theorem]{Remark}
\newtheorem{example}[theorem]{Example}

\makeatletter
 
 \@addtoreset{equation}{section}
\makeatother

\def\xx{\mathbf{x}}

\def\TT{\mathbb{T}}

\def\ZZ{\mathbb{Z}}

\def\Fcal{\mathcal{F}}

\title[Positive integer solutions to $(x+y)^2+(y+z)^2+(z+x)^2=12xyz$]{Positive integer solutions to $(x+y)^2+(y+z)^2+(z+x)^2=12xyz$}
\author{Yasuaki Gyoda}
\keywords{Diophantine equation, cluster algebra, mutation}
\subjclass[2020]{11D25,13F60}
\address{Graduate School of Mathematical Sciences, The University of Tokyo, 3-8-1 Komaba Meguro-ku Tokyo 153-8914, Japan}
\email{gyoda-yasuaki@g.ecc.u-tokyo.ac.jp}
\begin{document}
\maketitle
\begin{abstract}
    In this paper, we give a specific way of describing positive integer solutions of a Diophantine equation $(x+y)^2+(y+z)^2+(z+x)^2=12xyz$ and introduce a generalized cluster pattern behind it.
\end{abstract}

\renewcommand{\thefootnote}{\fnsymbol{footnote}}
\footnote[0]{
The author is grateful to Esther Banaian for her advice in adding Subsections 3.2 and 3.3 in the third arXiv version. The author also would like to thank Tomoki Nakanishi, Wataru Takeda, and Toshiki Matsusaka for their comments. The author also thank Matty van Son for pointing out a mistake. This work was supported by JSPS KAKENHI Grant number JP20J12675.}
\renewcommand{\thefootnote}{\arabic{footnote}}

\section{Introduction}
In this paper, we deal with a Diophantine equation
\begin{align}\label{Diophantine}
    (x+y)^2+(y+z)^2+(z+x)^2=12xyz.
\end{align}
The positive integer solutions to \eqref{Diophantine} include, for example, 
\[(1,1,1), (1,1,3), (1,3,13), (1,13,61), (3,13,217)\dots.\]
We describe all positive integer solutions to \eqref{Diophantine} in a combinatorial way. We give a tree $\mathbb T$ with triplets of positive integers as its vertices in the following steps.

\begin{itemize}
    \item [(1)] The root vertex is $(1,1,1)$ and its child is $(1,1,3)$, and its child is $(1,13,3)$.
    \item [(2)] The generation rule below $(1,13,3)$ is that a parent $(a,b,c)$ has the following two children:
\begin{itemize}
        \item [(i)] if $a$ is the maximal number in $(a,b,c)$, then $(a,b,c)$ has two children $(a,6ac-a-b-c,c)$ and $(a,b,6ab-a-b-c)$. 
        \item [(ii)] if $b$ is the maximal number in $(a,b,c)$, then $(a,b,c)$ has two children $(6bc-a-b-c,b,c)$ and $(a,b,6ab-a-b-c)$. 
        \item [(iii)] if $c$ is the maximal number in $(a,b,c)$, then $(a,b,c)$ has two children $(6bc-a-b-c,b,c)$ and $(a,6ac-a-b-c,c)$. 
\end{itemize}
\end{itemize}
The first few terms are as follows:
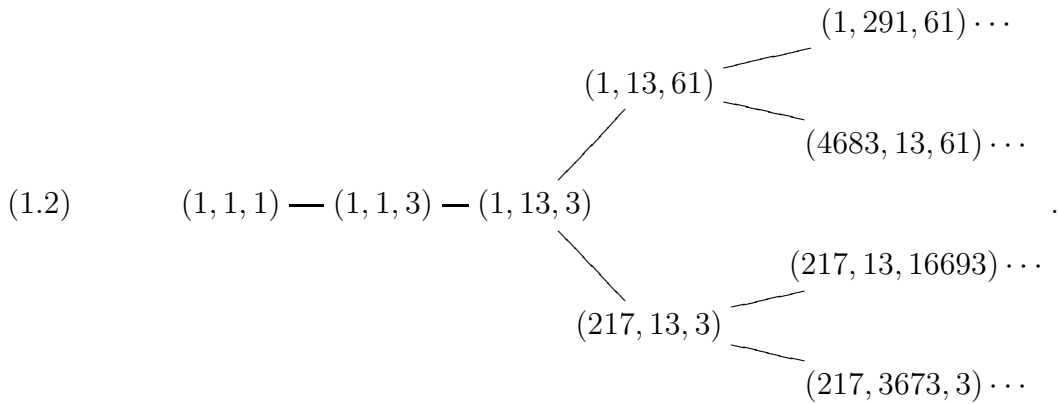
\begin{align}\label{tree}
\begin{xy}(0,0)*+{(1,1,1)}="0",(20,0)*+{(1,1,3)}="1",(40,0)*+{(1,13,3)}="2",(55,16)*+{(1,13,61)}="4",(55,-16)*+{(217,13,3)}="5", 
(90,24)*+{(1,291,61)\cdots}="6",(90,8)*+{(4683,13,61)\cdots}="7",(90,-8)*+{(217,13,16693)\cdots}="8",(90,-24)*+{(217,3673,3)\cdots}="9", \ar@{-}"0";"1"\ar@{-}"1";"2"\ar@{-}"2";"4"\ar@{-}"2";"5"\ar@{-}"4";"6"\ar@{-}"4";"7"\ar@{-}"5";"8"\ar@{-}"5";"9"
\end{xy}.
\end{align}
Note that the operation of changing $(a,b,c)$ to $(6bc-a-b-c,b,c)$ is an involution, that is, if we change a number at the same position twice, it returns to its original triple. We also remark that $(1,1,1), (1,1,3)$ and $(1,13,3)$ are also shifted by the operations $(a,b,c)\mapsto(a,6ac-a-b-c,c)$ or $(a,b,c)\mapsto(a,b,6ab-a-b-c)$, although they are not included in the definition.

The main result is the following theorem:

\begin{theorem}\label{Diophantinetheorem}
Every positive integer solution to \eqref{Diophantine} appears exactly once in $\mathbb T$ (up to order).
\end{theorem}

One of equations whose positive integer solution have the same structure as $\TT$ is the \emph{Markov Diophantine equation},

\begin{align}\label{Diophantine2}
    x^2+y^2+z^2=3xyz.
\end{align}

Consider a tree where $(1,1,3)$ is replaced by $(1,1,2)$, $(1,13,3)$ is replaced by $(1,5,2)$, and $6ab-a-b-c,6bc-a-b-c,6ac-a-b-c$ are replaced by $3ab-c,3bc-a,3ac-b$ respectively in the rules that construct $\mathbb T$.
The first few terms are as follows:
\begin{align}\label{tree2}
\begin{xy}(0,0)*+{(1,1,1)}="0",(20,0)*+{(1,1,2)}="1",(40,0)*+{(1,5,2)}="2",(55,16)*+{(1,5,13)}="4",(55,-16)*+{(29,5,2)}="5", 
(90,24)*+{(1,34,13)\cdots}="6",(90,8)*+{(194,5,13)\cdots}="7",(90,-8)*+{(29,5,433)\cdots}="8",(90,-24)*+{(29,169,2)\cdots}="9", \ar@{-}"0";"1"\ar@{-}"1";"2"\ar@{-}"2";"4"\ar@{-}"2";"5"\ar@{-}"4";"6"\ar@{-}"4";"7"\ar@{-}"5";"8"\ar@{-}"5";"9"
\end{xy}.
\end{align}
This tree is called the \emph{Markov tree}. It is a known classical result that every positive integer solution to \eqref{Diophantine2} appears exactly once in the Markov tree (see \cite{aig}*{Section 3.1}, for example). This equation has been very well studied for many years (see \cite{aig}), and numbers appearing in the Markov tree are called \emph{Markov numbers}. In addition, the Markov Diophantine equation has the following famous open problem.

\begin{conjecture}[Markov's uniqueness conjecture]
Every Markov number appears exactly once as the maximal number of some positive integer solution to \eqref{Diophantine2}.
\end{conjecture}

In parallel with this conjecture, we can consider the following conjecture for the equation \eqref{Diophantine}.

\begin{conjecture}\label{conjecture}
Every number appearing in the tree $\mathbb T$ appears exactly once as the maximal number of some positive integer solution to \eqref{Diophantine}.
\end{conjecture}

In the background of this similarity, there are structures derived from cluster algebra theory. Cluster algebras are algebras with a certain combinatorial structure, and this structure has been applied in various fields (for more information about cluster algebra, see \cites{fzi,fziv,fzw}, etc). This structure also appears in the positive integer solutions to the two equations mentioned above. To be more precise, while the Markov Diophantine equation \eqref{Diophantine2} has a specific ``cluster pattern'', the equation \eqref{Diophantine} has a specific ``generalized cluster pattern''. For detail of a relation between the cluster pattern and the Markov Diophantine equation, see \cite{fzw}*{Section 3.4} (and also Remark \ref{remark} in this paper). In this paper, we will explain the generalized cluster pattern inducing the structure of positive integer solutions to the equation \eqref{Diophantine}. Furthermore, by using this structure, we give some combinatorial implication of positive integers appearing in the solutions based on some previous works \cites{chsh, bake, cs18} on cluster algebra.

\begin{remark}
Theorem \ref{Diophantine} was discovered by Esther Banaian independently of the author at about the same time. See also Banaian's doctoral dissertation \cite{Bana}.
\end{remark}

\section{Proof of Theorem \ref{Diophantinetheorem} and its corollaries}
We begin with the following proposition:
\begin{proposition}\label{inductive-solution}
If $(x,y,z)=(a,b,c)$ is a positive integer solution to \eqref{Diophantine}, then so are $(6bc-a-b-c,b,c),(a,6ac-a-b-c,c),(a,b,6ab-a-b-c)$.
\end{proposition}
\begin{proof}
We prove only that $(6bc-a-b-c,b,c)$ is a positive solution. First, we show that $(6bc-a-b-c,b,c)$ is a solution to \eqref{Diophantine}.
We have
\begin{align*}
    &(6 b c - a - b - c + b)^2 + (6 b c - a - b - c + c)^2 + (b + c)^2\\
    =&2 a^2 + 2 a b + 2 b^2 + 2 a c + 2 b c - 24 a b c - 12 b^2 c + 2 c^2 - 12 b c^2 + 72 b^2 c^2\\
    =&-2 a^2 - 2 a b - 2 a c + 72 b^2 c^2 - 12 b^2 c - 2 b^2 - 12 b c^2 - 2 b c - 2 c^2,
\end{align*}
and
\begin{align*}
    &12(6bc-a-b-c)bc\\
    =& 12(6bc-b-c)bc-12abc\\
    =&12(6-b-c)bc-((a+b)^2+(b+c)^2+(c+a)^2)\\
    =&-2 a^2 - 2 a b - 2 a c + 72 b^2 c^2 - 12 b^2 c - 2 b^2 - 12 b c^2 - 2 b c - 2 c^2.
\end{align*}
Therefore, we have
\begin{align*}
    (6 b c - a - b - c + b)^2 + (6 b c - a - b - c + c)^2 + (b + c)^2=12(6bc-a-b-c)bc.
\end{align*}
This is an equality substituting $x=6bc-a-b-c,y=b,z=c$ in \eqref{Diophantine}. 

Next, we prove the positivity. By \eqref{Diophantine}, we have
\begin{align}\label{a^2+b^2+ab}
    6bc-a-b-c= \dfrac{6abc-a^2-ab-ac}{a}=\dfrac{b^2+c^2+bc}{a}.
\end{align}
Therefore, the positivity follows from \eqref{a^2+b^2+ab}.
\end{proof}
The next step is to determine the solutions that contain two or more of the same number.
\begin{lemma}\label{singular}
In the positive integer solutions to \eqref{Diophantine}, the only solutions that contain two or more of the same number are $(1,1,1)$ and $(1,1,3)$ up to order.
\end{lemma}
\begin{proof}
Let $(a,b,c)$ be a positive integer solution to \eqref{Diophantine} that contains two or more of the same number. We can assume $a=b$ without loss of generality.
Then, by substituting $(a,a,c)$ for $(x,y,z)$ in \eqref{Diophantine}, we have
\[
2(a+c)^2+4a^2=12a^2c.
\]
Therefore, we have 
\[
c=-a+3a^2\pm a\sqrt{(3a-1)^2-3}.
\]
Let $k=3a-1$. In this case, since $a$ is an integer, $k^2-3$ must be written as $l^2$ using some integer $l$. Therefore, we have $3=k^2-l^2=(k+l)(k-l)$. Therefore, we have $k=\pm 2,l=\pm1$. Since $a> 0$, we have $k=2,a=1$, and $c=1,3$. This finishes the proof.
\end{proof}
Solutions that contain two or more of the same number, such as $(1,1,1)$ or $(1,1,3)$, are called \emph{singular}, and other positive integer solutions to \eqref{Diophantine} \emph{nonsingular}\footnote{These names ``singular" and ``non-singular" are based on Aigner's book \cite{aig}.}.

\begin{proposition}\label{nonsingular-induction}
Let $(x,y,z)=(a,b,c)$ be a nonsingular positive integer solution to \eqref{Diophantine}, and we assume that $a>b>c$. Then we have
\begin{itemize}
\item[(1)] $6ac-a-b-c>a(>c)$,
\item[(2)] $6ab-a-b-c>a(>b)$,
\item[(3)] $b>6bc-a-b-c$.
\end{itemize}
\end{proposition}
\begin{proof}
We prove (1). It suffices to show that $6ac-a-b-c>a$.
By (an analogue of) \eqref{a^2+b^2+ab}, we have
\begin{align*}
    6ac-a-b-c-a&=\dfrac{a^2+c^2+ac}{b}-a=\dfrac{a^2+c^2+ac-ab}{b}>\dfrac{a^2+c^2+ac-a^2}{b}\\&=\dfrac{c^2+ac}{b}>0.
\end{align*}
We can show (2) in the same way as (1). We will show (3).
We set
\begin{align*}
    f(x):=(x-a)(x-(6bc-a-b-c))=x^2-(6bc-b-c)x+(b^2+c^2+bc)
\end{align*}
(we remark that $a(6bc-a-b-c)=b^2+c^2+bc$ is used). It suffices to show that 
\[f(b)=3b^2-6b^2c+2bc+c^2<0.\] We consider a function from $\mathbb{R}^2$ to $\mathbb{R}$
\begin{align}\label{equation}
    g(y,z)=3y^2-6y^2z+2yz+z^2.
\end{align}
We remark that $g(b,c)=f(b)$.
First, we show that if $y$ and $z$ satisfy $y> 1$, $y>z$ and $g(y,z)=0$, then we have $z<1$. We have
\begin{align*}
    y=\dfrac{-z\pm z\sqrt{-2+6z}}{3-6z}
\end{align*}
by $g(y,z)=0$. When $\sqrt{-2+6z}\leq 1$, since $z\leq\dfrac{1}{2}<1$, we do not need to consider this case. We assume that $\sqrt{-2+6z}>1$. Then, by $3-6z<0$ and $y>0$, we have
\begin{align}\label{sign-determine}
    -z\pm z\sqrt{-2+6z}<0.
\end{align}
If the sign $\pm$ in \eqref{sign-determine} is $+$, then there is not $z$ satisfying \eqref{sign-determine}, thus the sign is $-$.
Therefore, it suffices to consider
  \[\dfrac{-z- z\sqrt{-2+6z}}{3-6z}>1.\]
Since $3-6z<0$, we have 
\begin{align}\label{condition1}
5z-z\sqrt{-2+6z}-3<0,
\end{align}
thus we get $z>3,\dfrac{1}{3}\leq z<1,z=0$. By $y>z$, we have 
\begin{align}\label{condition2}
-z- z\sqrt{-2+6z}<z(3-6z),
\end{align}
thus we get $\dfrac{1}{3}\leq z<1$. By taking the intersection of solutions to \eqref{condition1},\eqref{condition2}, and $3-6z<0$, we have $\dfrac{1}{2}<z<1$, in particular, $z<1$. Next, we will show $f(b)<0$. When $y=2,z=1$, we have
\begin{align}
    g(2,1)=12-24+4+1=-7<0.
\end{align}
Because of the continuity of $g$ and the fact that $g(y,z)=0, y> 1$ and $y>z$ implies $z<1$, we have $g(b,c)=f(b)<0$.
\end{proof}
Now, we will show Theorem \ref{Diophantinetheorem}.

\begin{proof}[Proof of Theorem \ref{Diophantinetheorem}]
By Proposition \ref{inductive-solution} and the fact that $(x,y,z)=(1,1,1)$ is a positive integer solution to \eqref{Diophantine}, all vertices in $\mathbb T$ are positive integer solutions to \eqref{Diophantine}. Suppose that $(x,y,z)=(a,b,c)$ is a nonsingular positive integer solution to \eqref{Diophantine}. We assume $a$ is the maximal number of $(a,b,c)$. Then, by Proposition \ref{nonsingular-induction}, we have a positive integer solution $(x,y,z)=(6bc-a-b-c,b,c)$ to \eqref{Diophantine} whose maximal number is smaller than $a$. This process can be continued as long as the solution is nonsingular. However, since the solutions that appear in this operation are always positive integer solutions, a singular solution will appear in a finite number of the operations. By Lemma \ref{singular}, when a nonsingular solution changes to a singular solution, the nonsingular solution is $(1,13,3)$ and the singular solution is $(1,1,3)$. Therefore, by following this operation in reverse, $(a,b,c)$ is contained in the vertices of the tree $\mathbb T$. We prove the uniqueness. If not, we see that $(1,13,3)$ is not unique by repeating above operations. However, Proposition 2.3 and the generation rule of $\TT$ show that each child has a larger maximum than its parent, and the maximal number of $(a,b,c)$ except for $(1,1,1), (1,1,3), (1,13,3)$ written in $\eqref{tree}$ is larger than $13$. This is a contradiction.
\end{proof}

As in the Markov Diophantine equation \eqref{Diophantine2}, there are several corollaries that can be established.

\begin{corollary}
For any positive integer solution $(x,y,z)=(a,b,c)$ to \eqref{Diophantine}, all pairs in $a,b,c$ are relatively prime.
\end{corollary}
\begin{proof}
The claim is true for $(a,b,c)=(1,1,1),(1,1,3),(1,13,3)$. 
We prove only that $a$ and $b$ are relatively prime. By transforming \eqref{Diophantine} as \[z^2=6xyz-x^2-y^2-xy-yz-zx,\] and substituting $(x,y,z)=(a,b,c)$, if $a,b$ have a common divisor $d\neq 1$, then we see that $c$ can be divided by a prime divisor $d'$ of $d$. Thus, $d'$ is a common divisor of $a,b,c$. Therefore, by Proposition \ref{inductive-solution}, the neighbor $(a',b',c')$ of $(a,b,c)$ on the tree $\mathbb T$ whose maximal number is smaller than $\max\{a,b,c\}$ has the common divisor $d'$. By repeating this operation, we see that $d'$ is a common divisor of $(1,13,3)$. Thus, we must $d'=1$. This is a contradiction. Therefore, we have $d=1$.
\end{proof}

\begin{corollary}\label{cor}
Every number appearing in the tree $\mathbb T$ appears as the maximal number of some positive integer solution to \eqref{Diophantine}.
\end{corollary}

\begin{proof}
Let $n$ be a number appearing in $\mathbb T$. When $n=1$ and $3$, they are the maximal numbers of $(1,1,1)$ and $(1,1,3)$, respectively. We assume $n\geq 13$. We take a positive integer solution $(x,y,z)=(a,b,c)$ containing $n$. We assume that $a> b>c$. If $n=a$, then we are done. If $n=b$, then $n$ is the maximal number in the neighbor of $(a,b,c)$ in the tree $\mathbb T$ obtained by swapping $a$ by Proposition \ref{inductive-solution}. If $n=c$, as we traverse the neighbors with smaller maxima, $n$ becomes the second largest (otherwise, we get to $(1,13,3)$, and we have $n=1$, contradicting the assumption). Therefore, this case is attributed to the $n=b$ case.
\end{proof}

Conjecture \ref{conjecture} is a strong version of Corollary \ref{cor}.

\section{Generalized cluster pattern of positive integer solutions to \eqref{Diophantine}}

In this section, we introduce a generalized cluster pattern and we verify that the tree $\mathbb T$ is realized as a special case of a generalized cluster pattern. Moreover, we discuss the relation between the geometric realization of the cluster pattern and integer solutions to \eqref{Diophantine} via snake graphs.

\subsection{Generalized cluster pattern}

We start with recalling definitions of seed mutations and generalized cluster patterns according to \cites{chsh,nak15}\footnote{In \cites{chsh,nak15}, seeds and their mutations are defined in a version with $y$-variables (coefficients) or $z$-variables, but in this paper, these are not necessary and have been omitted}.
Let $n\in \ZZ_{\geq0}$ and $\Fcal$ a rational function field of $n$ indeterminates.
A \emph{labeled seed} is a pair $(\mathbf{x},B)$, where
\begin{itemize}
\item $\mathbf{x}=(x_1, \dots, x_n)$ is an $n$-tuple of elements of $\mathcal F$ forming a free generating set of $\mathcal F$,
\item $B=(b_{ij})$ is an $n \times n$ integer matrix which is \emph{skew-symmetrizable}, that is, there exists a positive integer diagonal matrix $S$ such that $SB$ is skew-symmetric. Also, we call $S$ a \emph{skew-symmetrizer} of $B$.
\end{itemize}

We say that $\xx$ is a \emph{cluster}, and  we refer to $x_i$ and $B$ as the \emph{cluster variables} and the \emph{exchange matrix}, respectively. Furthermore, we fix $D=\text{diag} (d_1,\dots, d_n)$ (called the \emph{associated diagonal matrix} with $(\xx,B)$), that is a positive integer diagonal matrix of rank $n$.

For an integer $b$, we use the notation $[b]_+=\max(b,0)$. 
Let $(\mathbf{x}, B)$ be a labeled seed, and let $k \in\{1,\dots, n\}$. The \emph{seed mutation $\mu_k$ in direction $k$} transforms $(\mathbf{x},B)$ into another labeled seed $\mu_k(\mathbf{x}, B)=(\mathbf{x'}, B')$ defined as follows:
\begin{itemize}
\item The entries of $B'=(b'_{ij})$ are given by
\begin{align} \label{eq:matrix-mutation}
b'_{ij}=\begin{cases}-b_{ij} &\text{if $i=k$ or $j=k$,} \\
b_{ij}+d_k\left(\left[ b_{ik}\right] _{+}b_{kj}+b_{ik}\left[ -b_{kj}\right]_+\right) &\text{otherwise.}
\end{cases}
\end{align}
\item The cluster variables $\mathbf{x'}=(x'_1, \dots, x'_n)$ are given by
\begin{align}\label{eq:x-mutation}
x'_j=\begin{cases}\dfrac{\left(\mathop{\prod}\limits_{i=1}^{n} x_i^{[-b_{ik}]_+}\right)^{d_k}\mathop{\sum}\limits_{s=0}^{d_k}\left(\mathop{\prod}\limits_{i=1}^{n} x_i^{b_{ik}}\right)^s}{x_k} &\text{if $j=k$,}\\
x_j &\text{otherwise.}
\end{cases}
\end{align}
\end{itemize}

Let $\mathbb{T}_n$ be the \emph{$n$-regular tree} whose edges are labeled by the numbers $1, \dots, n$ such that the $n$ edges emanating from each vertex have different labels. We write
$\begin{xy}(0,0)*+{t}="A",(10,0)*+{t'}="B",\ar@{-}^k"A";"B" \end{xy}$
to indicate that vertices $t,t'\in \mathbb{T}_n$ are joined by an edge labeled by $k$. We fix an arbitrary vertex $t_0\in \TT_n$, which is called the \emph{rooted vertex}.
A \emph{generalized cluster pattern} is an assignment of a labeled seed $\Sigma_t=(\xx_t,B_t)$ to every vertex $t\in \mathbb{T}_n$ such that the labeled seeds $\Sigma_t$ and $\Sigma_{t'}$ assigned to the endpoints of any edge
$\begin{xy}(0,0)*+{t}="A",(10,0)*+{t'}="B",\ar@{-}^k"A";"B" \end{xy}$
are obtained from each other by the seed mutation in direction $k$. We denote by $CP\colon t\mapsto \Sigma_t$ this assigment.

The degree $n$ of the regular tree $\TT_n$ is called the \emph{rank} of a generalized cluster pattern $CP$.
We also denote by $CP_{(\xx,B,D)}$ a generalized cluster pattern with the initial seed $\Sigma_{t_0}=(\xx,B)$ with the associated diagonal matrix $D$.

\begin{remark}
When $D=I_n$ (the identity matrix), a generalized cluster pattern coincides with a \emph{cluster pattern} defined in \cite{fziv}. 
\end{remark}

We will now look at an example that gives a positive integer solution to \eqref{Diophantine}. We set $B=\Delta:=\begin{bmatrix}0&1&-1\\-1&0&1\\1&-1&0\end{bmatrix}$, and $D=2I_3=\text{diag}(2,2,2)$. Then, for any $k\in\{1,2,3\}$, we have 
\begin{align*}
    \mu_k(\pm \Delta)=\mp \Delta.
\end{align*}
Therefore, the mutation of cluster variables are 

\begin{align*}\label{eq:x-mutation2}
x'_j=\begin{cases}\dfrac{x_\ell^2(1+x_\ell^{-1}x_m+x_\ell^{-2}x_m^2)}{x_k}=\dfrac{x_\ell^2+x_\ell x_m+x_m^2}{x_k} &\text{if $j=k$,}\\
x_j &\text{otherwise,}
\end{cases}
\end{align*}
where $\{k,\ell,m\}=\{1,2,3\}$.
This oparation coincides with the operation \[(x_k,x_\ell,x_m)\mapsto (6x_\ell x_m-x_k-x_\ell-x_m,x_\ell,x_m)\]
under the condition \eqref{Diophantine}, that is,
\[(x_k+x_\ell)^2+(x_\ell+x_m)^2+(x_m+x_k)^2=12x_kx_\ell x_m.\]

Therefore, by considering subtree $\TT'_3$ of $\TT_3$ (see Figure \ref{T'3}), we have the following theorem:
\begin{theorem}\label{clusterstructure}
We set $B=\Delta,\xx=(x_1,x_2,x_3)$, and $D=2I_3$. Then, the restriction of the generalized cluster pattern $CP_{(\xx,\Delta,2I_3)}$ with a substitution $x_1=x_2=x_3=1$ to $\mathbb T'_3$ gives the tree $\mathbb T$ (where we ignore exchange matrices in $CP_{(\xx,\Delta,2I_3)}$ and consider only clusters). 
\end{theorem}

\begin{remark}\label{remark}
In Theorem \ref{clusterstructure}, by setting $(B,D)=(2\Delta,I_3)$, we get the Markov tree \eqref{tree2} instead of $\TT$. 
\end{remark}

\begin{figure}[ht]
\caption{The tree $\mathbb T'_3$}
    \label{T'3}
\[\begin{xy}(0,0)*+{t_0}="0",(20,0)*+{t_1}="1",(40,0)*+{t_2}="2",(55,16)*+{t_3}="4",(55,-16)*+{t_4}="5", 
(75,24)*+{t_5\cdots}="6",(75,8)*+{t_6\cdots}="7",(75,-8)*+{t_7\cdots}="8",(75,-24)*+{t_8\cdots}="9", \ar@{-}^{3}"0";"1"\ar@{-}^{3}"2";"4"\ar@{-}_{1}"2";"5"\ar@{-}^{2}"4";"6"\ar@{-}_{1}"4";"7"\ar@{-}^{3}"5";"8"\ar@{-}_{2}"5";"9"\ar@{-}^{2}"1";"2"
\end{xy}.\]
\end{figure}
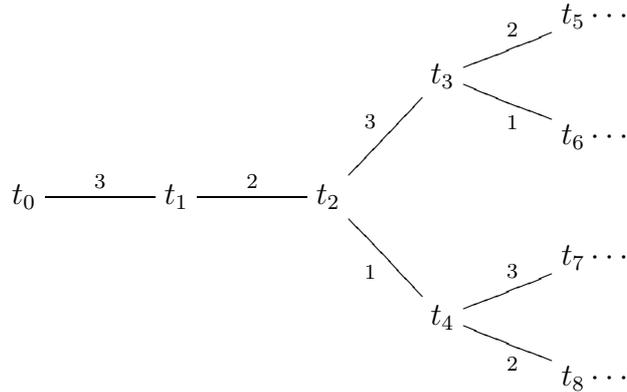

In the Markov Diophantine equation \eqref{Diophantine2}, numbers of positive integer solutions containing $1$ are all Fibonacci numbers. We give a corollary on positive integer solutions to \eqref{Diophantine} containing 1.

Under the setting $B=\Delta$ and $D=2I_3$, we consider banning a mutation in direction 1. Substituting 1 for the invariant first component of clusters, the mutation rule is 
\begin{align*}\label{eq:x-mutation2}
x'_j=\begin{cases}\dfrac{x_\ell^2+x_\ell+1}{x_k} &\text{if $j=k$,}\\
x_j &\text{otherwise,}
\end{cases}
\end{align*}
where $\{k,\ell\}=\{2,3\}$. This equation coincides with the formula in \cite{oeis}*{A101368}. Therefore, we have the following corollary.

\begin{corollary}
Let $(1,b,c)$ be a positive integer solution to \eqref{Diophantine}. Then, $b$ divides $1+c+c^2$ and $c$ divides $1+b+b^2$.
\end{corollary}
\subsection{Realization of integer solutions by perfect matching of snake graph} In \cite{chsh}, Chekhov and Shapiro gave a geometric realization of generalized cluster patterns when the numerator of the fraction in \eqref{eq:x-mutation} is trinomial. Here, we will focus only on the geometric realization corresponding to the cluster pattern giving the solutions to the equation \eqref{Diophantine}. For the general discussion, see \cites{chsh, feshtu}, etc.

An \emph{orbifold} is defined by the triplet $\mathcal O=(S,M,Q)$, where $S$ is a borderd Riemann surface, $M$ is a finite set of marked points in $S$, and $Q$ is a finite set of special points in $S$, called \emph{orbifold points}. Some Marked points belong to $\partial S$ and some belong to the interior of $S$. The latter are called \emph{punctures}. We assume $M\cap Q=\emptyset$ and $Q$ is included in the interior of $S$. Moreover, each orbifold point has an associated positive integer $p\geq 2$, called an \emph{order}. We define an \emph{arc} in $\mathcal O$ as a curve in $S$ with endpoints in $M\cup Q$ considered up to relative isotopy of $S\setminus M\cup Q$ modulo endpoints such that
\begin{itemize}
    \item both endpoints of $\gamma$ belong to $M$, or one endpoint belongs to $M$ and the other to $Q$,
    \item $\gamma$ has no intersections (however, it is possible that the endpoints are the same point),
    \item except for the endpoints, $\gamma$ and $M \cup Q \cup \partial S$ are disjoint,
    \item if $\gamma$ cuts out a monogon, then this monogon contains either a point of $M$ or at least two points of $Q$,
    \item $\gamma$ is not homotopic to a boundary segment of $S$.
\end{itemize}
An arc $\gamma$ is said to be a \emph{pending arc} if one of its endpoints belongs to $M$ and the other belongs to $Q$. Otherwise, that is, both endpoint of $\gamma$ belong to $M$, $\gamma$ is said to be a \emph{ordinary arc}. 

Next, we consider triangulating of orbifolds by arcs. Two arcs of $\mathcal O$ is \emph{compatible} if they satisfy the following conditions:
\begin{itemize}
    \item they do not intersect in the interior of $S$,
    \item if both $\gamma$ and $\gamma'$ are pending arcs, then the ends of $\gamma$ and $\gamma'$ that are orbifold points do not coincide. 
\end{itemize}

We define a \emph{triangulation} of $\mathcal O$ as the maximal set of pairwise compatible arcs in $\mathcal O$. Moreover, the operation of obtaining a new triangulation by replacing one edge of a triangulation with another edge is called a \emph{flip}. The triangulation obtained by flipping an arc is uniquely determined, and flipping the newly obtained arc again returns to the original triangulation. Let $T=(l_1,\dots,l_n)$ be a triangulation whose elements are labeled $1,\dots, n$. Then, we denote by $f_k(T)$ the flipped triangulation of $T$ obtained by replacing $l_k$.

\begin{remark}
The definition of order in this paper is based on \cite{bake} and is different from that of \cite{feshtu}. See also \cite{bake}*{Remark 3.4}.
\end{remark}


In this paper, we use an orbifold $\mathcal O_3$ consists of a sphere with one puncture and three orbifold points whose orders are all 3 and their triangulation. By moving orbifold points, we can assume that any triangulation of $\mathcal O_3$ is a diagram in Figure \ref{figure2}. 

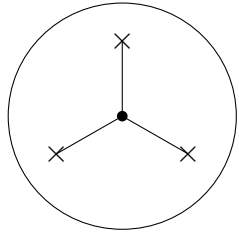
\begin{figure}[ht]
    \centering
    \caption{Sphere with one marked point and three orbifold points}
    \label{figure2}
\vspace{2mm}    
\begin{tikzpicture}
 \coordinate (0) at (0,0);
 \draw(0,0) circle (1.5);
 \node at (90:1) {$\times$};
\node at (-30:1) {$\times$};
\node at (210:1) {$\times$};
 \fill (0) circle (0.7mm); 
\draw(0,0) -- (90:1);
\draw(0,0) -- (-30:1);
\draw(0,0) -- (210:1);
\end{tikzpicture}
\end{figure}
We label these three pending arcs $1,2,3$. If the arc labeled 2 is exchanged with a flip, the resulting new triangulation is shown in Figure \ref{figure3}.
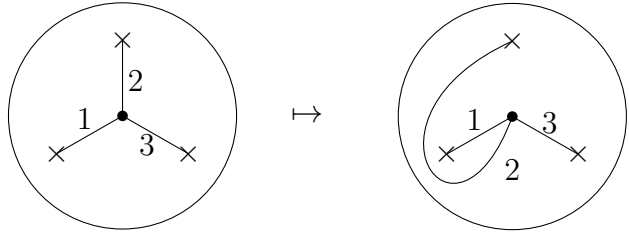
\begin{figure}[ht]
    \centering
    \caption{Flip of arc labeled 2}
    \label{figure3}
\vspace{2mm}    
\begin{tikzpicture}
 \coordinate (0) at (0,0);
 \draw(0,0) circle (1.5);
 \node at (90:1) {$\times$};
\node at (-30:1) {$\times$};
\node at (210:1) {$\times$};
\node at (185:0.5) {$1$};
\node at (70:0.5) {$2$};
\node at (-50:0.5) {$3$};
 \fill (0) circle (0.7mm); 
\draw(0,0) -- (90:1);
\draw(0,0) -- (-30:1);
\draw(0,0) -- (210:1);
\end{tikzpicture}
\hspace{3mm}
\begin{tikzpicture}[baseline=-1.5cm]
\node at (0:0) {$\mapsto$};
\end{tikzpicture}
\begin{tikzpicture}[baseline=-1.5cm]
 \coordinate (0) at (0,0);
 \draw(0,0) circle (1.5);
 \node at (90:1) {$\times$};
\node at (-30:1) {$\times$};
\node at (210:1) {$\times$};
 \fill (0) circle (0.7mm); 
\draw(0,0) -- (-30:1);
\draw(0,0) -- (210:1);
\node at (185:0.5) {$1$};
\node at (-90:0.7) {$2$};
\node at (-10:0.5) {$3$};
\draw(0:0) .. controls (250:2.2) and (180:2.2) .. (90:1);
\end{tikzpicture}
\end{figure}
We consider constructing \emph{snake graphs} from arcs of $\mathcal O_3$. Snake graphs of arcs in orbifolds are introduced by Banaian and Kelley in \cite{bake}\footnote{In \cite{bake}, snake graphs are introduced for a class of orbifolds without punctures, but exceptionally for closed surfaces with only one puncture where no tagged arc appears, it is possible to construct snake graphs using the same procedure as in this paper.}. In this paper, we present a method of constructing snake graphs of arcs in $\mathcal O_3$ by using ``unfolding". This method is given by Yurikusa in \cites{yur-expansion-A,y} in the case of surfaces without orbifold point, and we generalize it. Due to constructing the snake graph, all pending arcs are replaced by monogons surrounding the orbifold point (Figure \ref{figure4}). 
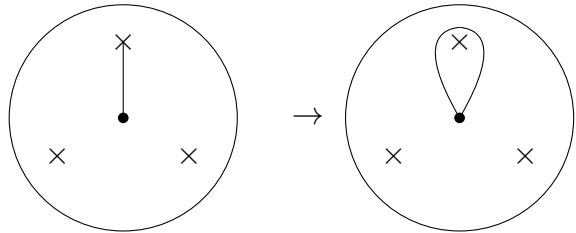
\begin{figure}[ht]
    \centering
    \caption{Replacement of pending arc}
    \label{figure4}
\vspace{2mm}    
\begin{tikzpicture}
 \coordinate (0) at (0,0);
 \draw(0,0) circle (1.5);
 \node at (90:1) {$\times$};
\node at (-30:1) {$\times$};
\node at (210:1) {$\times$};
 \fill (0) circle (0.7mm); 
\draw(0,0) -- (90:1);
\end{tikzpicture}
\hspace{3mm}
\begin{tikzpicture}[baseline=-1.5cm]
\node at (0:0) {$\to$};
\end{tikzpicture}
\begin{tikzpicture}
 \coordinate (0) at (0,0);
 \draw(0,0) circle (1.5);
\node at (90:1) {$\times$};
\node at (-30:1) {$\times$};
\node at (210:1) {$\times$};
 \fill (0) circle (0.7mm); 
\draw(0,0) .. controls (120:1.2) and (100:1.2) .. (90:1.2);
\draw(0,0) .. controls (60:1.2) and (80:1.2) .. (90:1.2);
\end{tikzpicture}
\end{figure}
First, we set a initial triangulation $T_0=(l_1,l_2,l_3)$ of $\mathcal{O}_3$. We can assume that the triangulation $T_0$ is the one shown in Figure \ref{figure5}.
\begin{figure}[ht]
    \centering
    \caption{Triangulation $T_0$}
    \label{figure5}
\begin{tikzpicture}
 \coordinate (0) at (0,0);
 \draw(0,0) circle (1.5);
 \node at (90:1) {$\times$};
\node at (-30:1) {$\times$};
\node at (210:1) {$\times$};
\node at (180:1) {$1$};
\node at (60:1) {$2$};
\node at (-60:1) {$3$};
 \fill (0) circle (0.7mm); 
\draw(0,0) .. controls (120:1.2) and (100:1.2) .. (90:1.2);
\draw(0,0) .. controls (60:1.2) and (80:1.2) .. (90:1.2);
\draw(0,0) .. controls (0:1.2) and (-20:1.2) .. (-30:1.2);
\draw(0,0) .. controls (-60:1.2) and (-40:1.2) .. (-30:1.2);
\draw(0,0) .. controls (180:1.2) and (200:1.2) .. (210:1.2);
\draw(0,0) .. controls (240:1.2) and (220:1.2) .. (210:1.2);
\end{tikzpicture}
\end{figure}
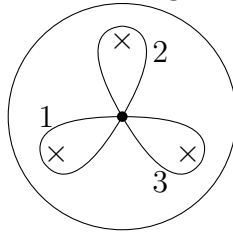
Let $\gamma$ be an arc in $\mathcal O_3$. We note that $\gamma$ is a pending arc, and in current situation, a monogon surrounding the orbifold point (from now on, arcs will refer to these replaced monogons). We set an orientation of $\gamma$. Since $T_0$ is a triangulation, if $\gamma$ is neither $l_1,l_2$ nor $l_3$, then $\gamma$ intersects one of $l_1,l_2$ or $l_3$ at least once except for its endpoints. Here, the number of intersections with $l_1$, $l_2$, and $l_3$ of the $\gamma$ should be minimized. We construct the \emph{pre-snake graph} $\tilde{G}_{T_0,\gamma}$ by the following way:
\begin{itemize}
    \item If $\gamma$ is $l_1,l_2$ or $l_3$, then $\tilde{G}_{T_0,\gamma}$ is an edge labeled $1,2$, or $3$, respectively, 
    \item If $\gamma$ is neither $l_1,l_2$ nor $l_3$, $\tilde{G}_{T_0,\gamma}$ is a figure in which the triangles through which $\gamma$ passes are pasted together. We remark that a monogon enclosed a orbifold point is regarded as a triangle in which all sides are curves that make up the monogon. When $\gamma$ passes through the monogon clockwise, the next triangle is pasted on the right side of the triangle corresponding to the monogon, and when passing through the monogon counterclockwise, the next triangle is pasted to the left side of the triangle corresponding to the monogon. See Figure \ref{figure6}. We label each edge of $\tilde{G}_{T_0,\gamma}$ with labels 1, 2, and 3 corresponding to $l_1$, $l_2$, and $l_3$.
\end{itemize}
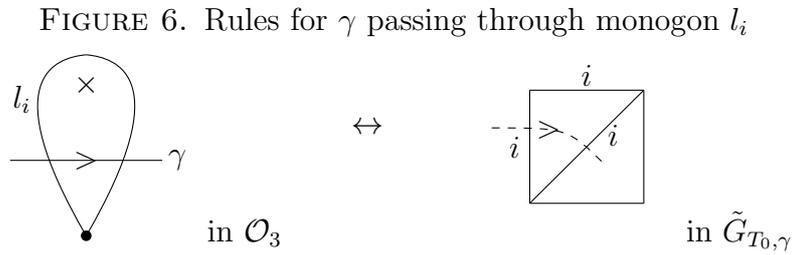
\begin{figure}[ht]
    \centering
    \caption{Rules for $\gamma$ passing through monogon $l_i$}
    \label{figure6}
\vspace{2mm}    
\begin{tikzpicture}
 \coordinate (0) at (0,0);
\node at (90:2) {$\times$};
 \fill (0) circle (0.7mm); 
\draw(0,0) .. controls (120:2.4) and (100:2.4) .. (90:2.4);
\draw(-1,1)--(1,1);
\node at (0,1) {$>$};
\node at (1.2,1) {$\gamma$};
\draw(0,0) .. controls (60:2.4) and (80:2.4) .. (90:2.4);
\node at (115:2) {$l_i$};
\end{tikzpicture}
in $\mathcal O_3$
\hspace{0.5cm}
\begin{tikzpicture}[baseline=-1.5cm]
\node at (0:0) {$\leftrightarrow$};
\end{tikzpicture}
\hspace{1cm}
\begin{tikzpicture}[baseline=-2cm]
\coordinate (1) at (0,0);
\coordinate (2) at (1.5,0);
\coordinate (3) at (0,-1.5);
\coordinate (4) at (1.5,-1.5);
\draw(1)--(2);
\draw(2)--(3);
\draw(3)--(1);
\draw(4)--(2);
\draw(4)--(3);
\node at (0.75,0.2) {$i$};
\node at (-0.2,-0.75) {$i$};
\node at (1.1,-0.6) {$i$};
\draw[dashed](-0.5,-0.5) .. controls (0.5,-0.5) and (0.5,-0.5) .. (1,-1);
\node at (0.25,-0.52) {$>$};
\end{tikzpicture}
\quad in $\tilde{G}_{T_0,\gamma}$
\end{figure}
\begin{example}\label{gamma-ex}
When the arc $\gamma$ is given as in the left-hand side of Figure \ref{figure7}, $\tilde{G}_{T_0,\gamma}$ is given as in the figure on the right-hand side of Figure \ref{figure7}.
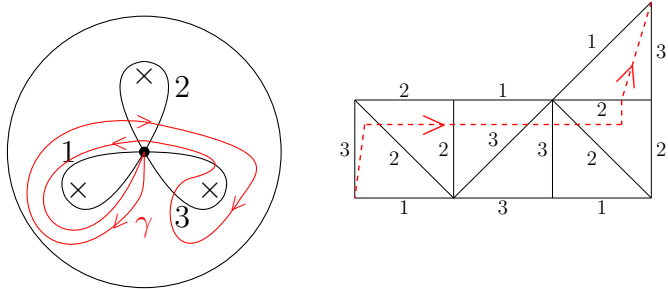
\begin{figure}[ht]
    \centering
    \caption{Arc $\gamma$ and $\tilde{G}_{T_0,\gamma}$}
    \label{figure7}
\begin{tikzpicture}[baseline=2cm]
 \coordinate (0) at (0,0);
 \draw(0,0) circle (1.8);
 \node at (90:1) {$\times$};
\node at (-30:1) {$\times$};
\node at (210:1) {$\times$};
\node at (180:1) {$1$};
\node at (60:1) {$2$};
\node at (-60:1) {$3$};
\node [red]at (-90:1) {$\gamma$};
 \fill (0) circle (0.7mm); 
\draw(0,0) .. controls (120:1.2) and (100:1.2) .. (90:1.2);
\draw(0,0) .. controls (60:1.2) and (80:1.2) .. (90:1.2);
\draw(0,0) .. controls (0:1.2) and (-20:1.2) .. (-30:1.2);
\draw(0,0) .. controls (-60:1.2) and (-40:1.2) .. (-30:1.2);
\draw(0,0) .. controls (180:1.2) and (200:1.2) .. (210:1.2);
\draw(0,0) .. controls (240:1.2) and (220:1.2) .. (210:1.2);
\draw[red](0:0) .. controls (250:2.5) and (180:2.5) .. (90:0.15);
\draw[red](-0.2,-0.8) .. controls (235:2.8) and (165:2.5) .. (90:0.4);
\draw[red](-0.2,-0.8) .. controls (0,-0.55) and (0,-0.55) .. (0,0);
\draw[red](90:0.15) .. controls (0:1) and (-10:1.2) .. (-30:0.7);
\draw[red](90:0.4) .. controls (0:2) and (-10:1.5) .. (-30:1.4);
\draw[red](-30:0.7) .. controls (-80:0.5) and (-80:2) .. (-30:1.4);
\node [red]at (-110:1) {\rotatebox{-130}{\tiny\text{$>$}}};
\node [red]at (90:0.4) {\rotatebox{0}{\tiny\text{$>$}}};
\node [red]at (-30:1.4) {\rotatebox{-130}{\tiny\text{$>$}}};
\node [red]at (160:0.35) {\rotatebox{-180}{\tiny\text{$>$}}};
\end{tikzpicture}
\hspace{2mm}
\rotatebox{-90}{\scalebox{0.65}{\begin{tikzpicture}
\coordinate(1) at (0,0){}; 
\coordinate(2) at (0,-2){}; 
\coordinate(3) at (0,-4){}; 
\coordinate(5) at (2,0){}; 
\coordinate(6) at (2,-2){}; 
\coordinate(7) at (2,-4){}; 
\coordinate(8) at (0,-6){}; 
\coordinate(9) at (2,-6){}; 
\coordinate(10) at (-2,0){}; 
\coordinate(11) at (-2,-2){}; 
\draw(1) to (5);
\draw(1) to (3);
\draw(3) to (7);
\draw(2) to (6);
\draw(5) to (2);
\draw(2) to (7);
\draw(5) to (9);
\draw(3) to (8);
\draw(8) to (9);
\draw(8) to (7);
\draw(2) to (10);
\draw(1) to (10);
\node at (1,-6.2){\rotatebox{90}{3}};
\node at (2.2,-5){\rotatebox{90}{1}};
\node at (1.2,-5.2){\rotatebox{90}{2}};
\node at (-0.2,-5){\rotatebox{90}{2}};
\node at (1,-4.2){\rotatebox{90}{2}};
\node at (0.8,-3.2){\rotatebox{90}{3}};
\node at (-0.2,-3){\rotatebox{90}{1}};
\node at (2.2,-3){\rotatebox{90}{3}};
\node at (1,-2.2){\rotatebox{90}{3}};
\node at (1.2,-1.2){\rotatebox{90}{2}};
\node at (2.2,-1){\rotatebox{90}{1}};
\node at (1,0.2){\rotatebox{90}{2}};
\node at (0.2,-1){\rotatebox{90}{2}};
\node at (-1,0.2){\rotatebox{90}{3}};
\node at (-1.2,-1.2){\rotatebox{90}{1}};
\draw[dashed,red,thick](-2,0)--(0,-0.6);
\draw[dashed,red,thick](0,-0.6)--(0.5,-0.6);
\draw[dashed,red,thick](0.5,-0.6)--(0.5,-5.8);
\draw[dashed,red,thick](0.5,-5.8)--(2,-6);
\node [red]at (0.5,-4.4) {\rotatebox{90}{\LARGE \text{$>$}}};
\node [red]at (-0.5,-0.45) {\rotatebox{160}{\LARGE \text{$>$}}};
\end{tikzpicture}}}
\end{figure}
\end{example}
In the pre-snake graph of $\mathcal O_3$, two types of triangles, a triangle labeled with the same number and a triangle labeled 1,2,3 counterclockwise, appear alternately.
Next, we define the \emph{snake graph} $G_{T_0,\gamma}$ as a graph from a pre-snake graph $\tilde{G}_{T_0,\gamma}$ by the following way:
\begin{itemize}
    \item [(1)] We unfold each triangle of $\tilde{G}_{T_0,\gamma}$, except for first and last ones, along its boundary side (see Figure \ref{figure8}),
\item [(2)] we shape the figure obtained in (1) so that each of the two triangles forms a square tile,
\item[(3)] we remove the diagonal from each square tile of the figure formed in (2).
\end{itemize}
We note that if $\gamma$ is either $l_1$, $l_2$, or $l_3$, then we have $G_{T_0,\gamma}=\tilde{G}_{T_0,\gamma}$.
 \begin{figure}[ht]
   \caption{Unfolding triangle, where $a$ is boundary segment, while $b$ and $c$ are not}
   \label{figure8}
   \vspace{2mm}
\begin{tikzpicture}
\coordinate(1) at (0,0){}; 
\coordinate(2) at (1.5,0){}; 
\coordinate(3) at (3,0){}; 
\coordinate(5) at (0,-1){}; 
\coordinate(6) at (1,-1){}; 
\coordinate(7) at (2,-1){}; 
\coordinate(8) at (3,-1){}; 
\draw(1) to (3);
\draw(5) to (6);
\draw[red](6) to (7);
\draw(7) to (8);
\draw(2) to (6);
\draw(2) to (7);
\node[red] at (1.5,-1.2){$a$};
\node at (1,-0.5){$b$};
\node at (2,-0.5){$c$};
\end{tikzpicture}
\begin{tikzpicture}[baseline=-9mm]
\node at (0,0){$\xrightarrow[\text{along }a]{\text{unfolding}}$};
\end{tikzpicture}
\begin{tikzpicture}
\coordinate(1) at (0,0){}; 
\coordinate(2) at (1.5,0){}; 
\coordinate(3) at (3,0){}; 
\coordinate(5) at (0,-1){}; 
\coordinate(6) at (1,-1){}; 
\coordinate(7) at (2,-1){}; 
\coordinate(8) at (3,-1){}; 
\coordinate(9) at (0,-2){}; 
\coordinate(10) at (1.5,-2){}; 
\coordinate(11) at (3,-2){}; 
\draw(1) to (2);
\draw(5) to (6);
\draw[red](6) to (7);
\draw(7) to (8);
\draw(2) to (6);
\draw(2) to (7);
\draw(10)to (11);
\draw(10)to(6);
\draw(10)to (7);
\node[red] at (1.5,-1.2){$a$};
\node at (1,-0.5){$b$};
\node at (2,-0.5){$c$};
\node at (1,-1.5){$b$};
\node at (2,-1.5){$c$};
\end{tikzpicture}
\end{figure}
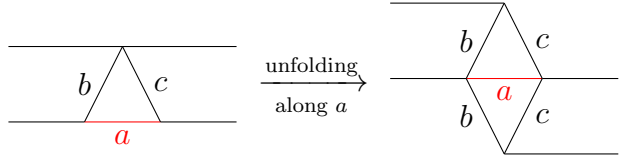
\begin{example}\label{ex1}
We consider $\gamma$ in Example \ref{gamma-ex}. Then, the snake graph $G_{T_0,\gamma}$ is given as in the last graph in Figure \ref{figure9}.
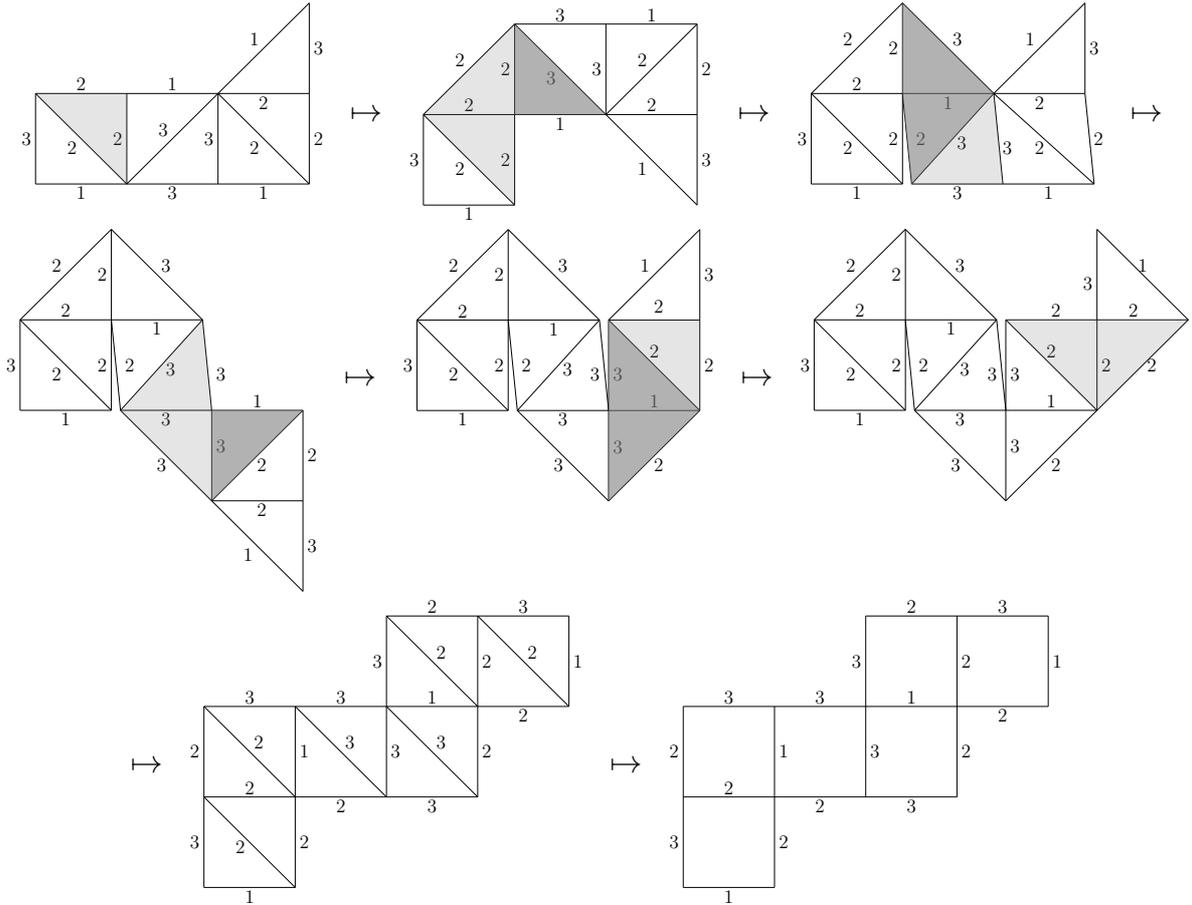
\begin{figure}[ht]
\centering
    \caption{From $\tilde{G}_{T_0,\gamma}$ to $G_{T_0,\gamma}$}
    \label{figure9}
\rotatebox{-90}{\scalebox{0.6}{\begin{tikzpicture}
\coordinate(1) at (0,0){}; 
\coordinate(2) at (0,-2){}; 
\coordinate(3) at (0,-4){}; 
\coordinate(5) at (2,0){}; 
\coordinate(6) at (2,-2){}; 
\coordinate(7) at (2,-4){}; 
\coordinate(8) at (0,-6){}; 
\coordinate(9) at (2,-6){}; 
\coordinate(10) at (-2,0){}; 
\coordinate(11) at (-2,-2){}; 
\draw(1) to (5);
\draw(1) to (3);
\draw(3) to (7);
\draw(2) to (6);
\draw(5) to (2);
\draw(2) to (7);
\draw(5) to (9);
\draw(3) to (8);
\draw(8) to (9);
\draw(8) to (7);
\draw(2) to (10);
\draw(1) to (10);
\node at (1,-6.2){\rotatebox{90}{3}};
\node at (2.2,-5){\rotatebox{90}{1}};
\node at (1.2,-5.2){\rotatebox{90}{2}};
\node at (-0.2,-5){\rotatebox{90}{2}};
\node at (1,-4.2){\rotatebox{90}{2}};
\node at (0.8,-3.2){\rotatebox{90}{3}};
\node at (-0.2,-3){\rotatebox{90}{1}};
\node at (2.2,-3){\rotatebox{90}{3}};
\node at (1,-2.2){\rotatebox{90}{3}};
\node at (1.2,-1.2){\rotatebox{90}{2}};
\node at (2.2,-1){\rotatebox{90}{1}};
\node at (1,0.2){\rotatebox{90}{2}};
\node at (0.2,-1){\rotatebox{90}{2}};
\node at (-1,0.2){\rotatebox{90}{3}};
\node at (-1.2,-1.2){\rotatebox{90}{1}};
\fill [gray,opacity=.2] (3)--(8)--(7)--(3);
\end{tikzpicture}}}
\begin{tikzpicture}[baseline=15mm]
\node at (0,0){$\mapsto$};
\end{tikzpicture}
\rotatebox{-90}{\scalebox{0.6}{\begin{tikzpicture}
\coordinate(1) at (0,0){}; 
\coordinate(2) at (0,-2){}; 
\coordinate(3) at (0,-4){}; 
\coordinate(5) at (-2,0){}; 
\coordinate(6) at (-2,-2){}; 
\coordinate(7) at (2,-4){}; 
\coordinate(8) at (0,-6){}; 
\coordinate(9) at (2,-6){}; 
\coordinate(10) at (-2,0){}; 
\coordinate(11) at (-2,-2){}; 
\coordinate(12) at (-2,-6){};
\coordinate(13) at (-2,-4){}; 
\coordinate(14) at (2,0){}; 
\draw(1) to (3);
\draw(3) to (7);
\draw(2) to (6);
\draw(5) to (13);
\draw(7) to (9);
\draw(3) to (8);
\draw(8) to (9);
\draw(8) to (7);
\draw(2) to (10);
\draw(1) to (10);
\draw(13) to (3);
\draw(13) to (2);
\draw(8) to (13);
\draw(14) to (1);
\draw(14) to (2);
\node at (1,-6.2){\rotatebox{90}{3}};
\node at (2.2,-5){\rotatebox{90}{1}};
\node at (1.2,-5.2){\rotatebox{90}{2}};
\node at (-1.2,-5.2){\rotatebox{90}{2}};
\node at (-0.2,-5){\rotatebox{90}{2}};
\node at (1,-4.2){\rotatebox{90}{2}};
\node at (-1,-4.2){\rotatebox{90}{2}};
\node at (-0.8,-3.2){\rotatebox{90}{3}};
\node at (0.2,-3){\rotatebox{90}{1}};
\node at (-2.2,-3){\rotatebox{90}{3}};
\node at (-1,-2.2){\rotatebox{90}{3}};
\node at (-1.2,-1.2){\rotatebox{90}{2}};
\node at (-2.2,-1){\rotatebox{90}{1}};
\node at (-1,0.2){\rotatebox{90}{2}};
\node at (-0.2,-1){\rotatebox{90}{2}};
\node at (1,0.2){\rotatebox{90}{3}};
\node at (1.2,-1.2){\rotatebox{90}{1}};
\fill [gray,opacity=.2] (13)--(8)--(7)--(13);
\fill [gray,opacity=.6] (13)--(2)--(3)--(13);
\end{tikzpicture}}}
\begin{tikzpicture}[baseline=15mm]
\node at (0,0){$\mapsto$};
\end{tikzpicture}
\rotatebox{-90}{\scalebox{0.6}{\begin{tikzpicture}
\coordinate(1) at (0,0){}; 
\coordinate(2) at (0,-2){}; 
\coordinate(3) at (0,-4){}; 
\coordinate(5) at (-2,0){}; 
\coordinate(6) at (-2,-2){}; 
\coordinate(7) at (2,-4){}; 
\coordinate(8) at (0,-6){}; 
\coordinate(9) at (2,-6){}; 
\coordinate(10) at (-2,0){}; 
\coordinate(11) at (-2,-2){}; 
\coordinate(12) at (-2,-6){};
\coordinate(13) at (-2,-4){}; 
\coordinate(14) at (2,0.2){}; 
\coordinate(15) at (2,-3.8){};
\coordinate(16) at (2,-1.8){};
\draw(1) to (3);
\draw(3) to (7);
\draw(7) to (9);
\draw(3) to (8);
\draw(8) to (9);
\draw(8) to (7);
\draw(2) to (10);
\draw(1) to (10);
\draw(13) to (3);
\draw(13) to (2);
\draw(8) to (13);
\draw(14) to (1);
\draw(14) to (2);
\draw(3) to (15);
\draw(2) to (16);
\draw(14) to (15);
\draw(2) to (15);
\node at (1,-6.2){\rotatebox{90}{3}};
\node at (2.2,-5){\rotatebox{90}{1}};
\node at (1.2,-5.2){\rotatebox{90}{2}};
\node at (-1.2,-5.2){\rotatebox{90}{2}};
\node at (-0.2,-5){\rotatebox{90}{2}};
\node at (1,-4.2){\rotatebox{90}{2}};
\node at (-1,-4.2){\rotatebox{90}{2}};
\node at (-1.2,-2.8){\rotatebox{90}{3}};
\node at (0.2,-3){\rotatebox{90}{1}};
\node at (2.2,-2.8){\rotatebox{90}{3}};
\node at (1.2,-1.7){\rotatebox{90}{3}};
\node at (1.2,-1){\rotatebox{90}{2}};
\node at (2.2,-0.8){\rotatebox{90}{1}};
\node at (1,0.3){\rotatebox{90}{2}};
\node at (0.2,-1){\rotatebox{90}{2}};
\node at (-1,0.2){\rotatebox{90}{3}};
\node at (-1.2,-1.2){\rotatebox{90}{1}};
\node at (1,-3.6){\rotatebox{90}{2}};
\node at (1.1,-2.7){\rotatebox{90}{3}};
\fill [gray,opacity=.6] (13)--(2)--(3)--(13);
\fill [gray,opacity=.6] (15)--(2)--(3)--(15);
\fill [gray,opacity=.2] (2)--(16)--(15)--(2);
\end{tikzpicture}}}
\begin{tikzpicture}[baseline=15mm]
\node at (0,0){$\mapsto$};
\end{tikzpicture}
\rotatebox{-90}{\scalebox{0.6}{\begin{tikzpicture}
\coordinate(1) at (0,0){}; 
\coordinate(2) at (0,-2){}; 
\coordinate(3) at (0,-4){}; 
\coordinate(5) at (-2,0){}; 
\coordinate(6) at (-2,-2){}; 
\coordinate(7) at (2,-4){}; 
\coordinate(8) at (0,-6){}; 
\coordinate(9) at (2,-6){}; 
\coordinate(10) at (-2,0){}; 
\coordinate(11) at (-2,-2){}; 
\coordinate(12) at (-2,-6){};
\coordinate(13) at (-2,-4){}; 
\coordinate(14) at (2,0.2){}; 
\coordinate(15) at (2,-3.8){};
\coordinate(16) at (2,-1.8){};
\coordinate(17) at (4,-1.8){};
\coordinate(18) at (4,0.2){};
\coordinate(19) at (6,0.2){};
\draw(2) to (3);
\draw(3) to (7);
\draw(7) to (9);
\draw(3) to (8);
\draw(8) to (9);
\draw(8) to (7);
\draw(13) to (3);
\draw(13) to (2);
\draw(8) to (13);
\draw(3) to (15);
\draw(2) to (16);
\draw(14) to (15);
\draw(2) to (15);
\draw(15) to (17);
\draw(16) to (17);
\draw(17) to (18);
\draw(14) to (18);
\draw(14) to (17);
\draw(18) to (19);
\draw(17) to (19);
\node at (1,-6.2){\rotatebox{90}{3}};
\node at (2.2,-5){\rotatebox{90}{1}};
\node at (1.2,-5.2){\rotatebox{90}{2}};
\node at (-1.2,-5.2){\rotatebox{90}{2}};
\node at (-0.2,-5){\rotatebox{90}{2}};
\node at (1,-4.2){\rotatebox{90}{2}};
\node at (-1,-4.2){\rotatebox{90}{2}};
\node at (-1.2,-2.8){\rotatebox{90}{3}};
\node at (0.2,-3){\rotatebox{90}{1}};
\node at (2.2,-2.8){\rotatebox{90}{3}};
\node at (1.2,-1.6){\rotatebox{90}{3}};
\node at (4.2,-0.7){\rotatebox{90}{2}};
\node at (1.8,-0.8){\rotatebox{90}{1}};
\node at (3,0.4){\rotatebox{90}{2}};
\node at (3.2,-0.7){\rotatebox{90}{2}};
\node at (5,0.4){\rotatebox{90}{3}};
\node at (5.2,-1){\rotatebox{90}{1}};
\node at (1,-3.6){\rotatebox{90}{2}};
\node at (1.1,-2.7){\rotatebox{90}{3}};
\node at (3.2,-2.9){\rotatebox{90}{3}};
\node at (2.8,-1.6){\rotatebox{90}{3}};
\fill [gray,opacity=.2] (2)--(16)--(15)--(2);
\fill [gray,opacity=.2] (17)--(16)--(15)--(17);
\fill [gray,opacity=.6] (17)--(16)--(14)--(17);
\end{tikzpicture}}}
\begin{tikzpicture}[baseline=20mm]
\node at (0,0){$\mapsto$};
\end{tikzpicture}
\rotatebox{-90}{\scalebox{0.6}{\begin{tikzpicture}
\coordinate(1) at (0,0){}; 
\coordinate(2) at (0,-2){}; 
\coordinate(3) at (0,-4){}; 
\coordinate(5) at (-2,0){}; 
\coordinate(6) at (-2,-2){}; 
\coordinate(7) at (2,-4){}; 
\coordinate(8) at (0,-6){}; 
\coordinate(9) at (2,-6){}; 
\coordinate(10) at (-2,0){}; 
\coordinate(11) at (-2,-2){}; 
\coordinate(12) at (-2,-6){};
\coordinate(13) at (-2,-4){}; 
\coordinate(14) at (2,0.2){}; 
\coordinate(15) at (2,-3.8){};
\coordinate(16) at (2,-1.8){};
\coordinate(17) at (4,-1.8){};
\coordinate(18) at (4,0.2){};
\coordinate(20) at (0,0.2){};
\coordinate(21) at (-2,0.2){};
\coordinate(22) at (0,-1.8){};
\draw(2) to (3);
\draw(3) to (7);
\draw(7) to (9);
\draw(3) to (8);
\draw(8) to (9);
\draw(8) to (7);
\draw(13) to (3);
\draw(13) to (2);
\draw(8) to (13);
\draw(3) to (15);
\draw(2) to (16);
\draw(14) to (15);
\draw(2) to (15);
\draw(15) to (17);
\draw(16) to (17);
\draw(14) to (17);
\draw(16) to (22);
\draw(14) to (20);
\draw(21) to (20);
\draw(20) to (22);
\draw(14) to (22);
\draw(21) to (22);
\node at (1,-6.2){\rotatebox{90}{3}};
\node at (2.2,-5){\rotatebox{90}{1}};
\node at (1.2,-5.2){\rotatebox{90}{2}};
\node at (-1.2,-5.2){\rotatebox{90}{2}};
\node at (-0.2,-5){\rotatebox{90}{2}};
\node at (1,-4.2){\rotatebox{90}{2}};
\node at (-1,-4.2){\rotatebox{90}{2}};
\node at (-1.2,-2.8){\rotatebox{90}{3}};
\node at (0.2,-3){\rotatebox{90}{1}};
\node at (2.2,-2.8){\rotatebox{90}{3}};
\node at (1.2,-1.6){\rotatebox{90}{3}};
\node at (1.2,-2.1){\rotatebox{90}{3}};
\node at (-0.3,-0.7){\rotatebox{90}{2}};
\node at (1.8,-0.8){\rotatebox{90}{1}};
\node at (1,0.4){\rotatebox{90}{2}};
\node at (3.2,-0.7){\rotatebox{90}{2}};
\node at (-1,0.4){\rotatebox{90}{3}};
\node at (-1.2,-1){\rotatebox{90}{1}};
\node at (1,-3.6){\rotatebox{90}{2}};
\node at (1.1,-2.7){\rotatebox{90}{3}};
\node at (3.2,-2.9){\rotatebox{90}{3}};
\node at (2.8,-1.6){\rotatebox{90}{3}};
\node at (0.7,-0.8){\rotatebox{90}{2}};
\fill [gray,opacity=.6] (17)--(16)--(14)--(17);
\fill [gray,opacity=.6] (22)--(16)--(14)--(22);
\fill [gray,opacity=.2] (22)--(14)--(20)--(22);
\end{tikzpicture}}}
\begin{tikzpicture}[baseline=20mm]
\node at (0,0){$\mapsto$};
\end{tikzpicture}
\rotatebox{-90}{\scalebox{0.6}{\begin{tikzpicture}[baseline=-50mm]
\coordinate(1) at (0,0){}; 
\coordinate(2) at (0,-2){}; 
\coordinate(3) at (0,-4){}; 
\coordinate(5) at (-2,0){}; 
\coordinate(6) at (-2,-2){}; 
\coordinate(7) at (2,-4){}; 
\coordinate(8) at (0,-6){}; 
\coordinate(9) at (2,-6){}; 
\coordinate(10) at (-2,0){}; 
\coordinate(11) at (-2,-2){}; 
\coordinate(12) at (-2,-6){};
\coordinate(13) at (-2,-4){}; 
\coordinate(14) at (2,0.2){}; 
\coordinate(15) at (2,-3.8){};
\coordinate(16) at (2,-1.8){};
\coordinate(17) at (4,-1.8){};
\coordinate(18) at (4,0.2){};
\coordinate(20) at (0,0.2){};
\coordinate(21) at (-2,0.2){};
\coordinate(22) at (0,-1.8){};
\coordinate(23) at (0,2.2){};
\draw(2) to (3);
\draw(3) to (7);
\draw(7) to (9);
\draw(3) to (8);
\draw(8) to (9);
\draw(8) to (7);
\draw(13) to (3);
\draw(13) to (2);
\draw(8) to (13);
\draw(3) to (15);
\draw(2) to (16);
\draw(14) to (15);
\draw(2) to (15);
\draw(15) to (17);
\draw(16) to (17);
\draw(14) to (17);
\draw(16) to (22);
\draw(14) to (20);
\draw(21) to (20);
\draw(20) to (22);
\draw(14) to (22);
\draw(21) to (23);
\draw(14) to (23);
\draw(20) to (23);
\node at (1,-6.2){\rotatebox{90}{3}};
\node at (2.2,-5){\rotatebox{90}{1}};
\node at (1.2,-5.2){\rotatebox{90}{2}};
\node at (-1.2,-5.2){\rotatebox{90}{2}};
\node at (-0.2,-5){\rotatebox{90}{2}};
\node at (1,-4.2){\rotatebox{90}{2}};
\node at (-1,-4.2){\rotatebox{90}{2}};
\node at (-1.2,-2.8){\rotatebox{90}{3}};
\node at (0.2,-3){\rotatebox{90}{1}};
\node at (2.2,-2.8){\rotatebox{90}{3}};
\node at (1.2,-1.6){\rotatebox{90}{3}};
\node at (1.2,-2.1){\rotatebox{90}{3}};
\node at (-0.2,-0.7){\rotatebox{90}{2}};
\node at (-0.2,1){\rotatebox{90}{2}};
\node at (1.8,-0.8){\rotatebox{90}{1}};
\node at (1,0.4){\rotatebox{90}{2}};
\node at (3.2,-0.7){\rotatebox{90}{2}};
\node at (-0.8,0){\rotatebox{90}{3}};
\node at (-1.2,1.2){\rotatebox{90}{1}};
\node at (1,-3.6){\rotatebox{90}{2}};
\node at (1.1,-2.7){\rotatebox{90}{3}};
\node at (3.2,-2.9){\rotatebox{90}{3}};
\node at (2.8,-1.6){\rotatebox{90}{3}};
\node at (0.7,-0.8){\rotatebox{90}{2}};
\node at (1,1.4){\rotatebox{90}{2}};
\fill [gray,opacity=.2] (22)--(14)--(20)--(22);
\fill [gray,opacity=.2] (23)--(14)--(20)--(23);
\end{tikzpicture}}}
\begin{tikzpicture}[baseline=20mm]
\node at (0,0){$\mapsto$};
\end{tikzpicture}
\rotatebox{0}{\scalebox{0.6}{\begin{tikzpicture}[baseline=0mm]
\coordinate(1) at (0,0){}; 
\coordinate(2) at (2,0){}; 
\coordinate(3) at (4,0){}; 
\coordinate(4) at (-4,-2){};
\coordinate(5) at (-2,-2){}; 
\coordinate(6) at (0,-2){}; 
\coordinate(7) at (2,-2){};
\coordinate(8) at (4,-2){};
\coordinate(9) at (-4,-4){};
\coordinate(10) at (-2,-4){}; 
\coordinate(11) at (0,-4){}; 
\coordinate(12) at (2,-4){};
\coordinate(13) at (-4,-6){}; 
\coordinate(14) at (-2,-6){}; 
\draw(1) to (3);
\draw(4) to (8);
\draw(9) to (12);
\draw(13) to (14);
\draw(3) to (8);
\draw(2) to (12);
\draw(1) to (11);
\draw(5) to (14);
\draw(4) to (13);
\draw(2) to (8);
\draw(1) to (7);
\draw(6) to (12);
\draw(5) to (11);
\draw(4) to (10);
\draw(9) to (14);
\node at (-3,-6.2){{1}};
\node at (-4.2,-5){{3}};
\node at (-1.8,-5){{2}};
\node at (-3.2,-5.1){{2}};
\node at (-3,-3.8){{2}};
\node at (-4.2,-3){{2}};
\node at (-2.8,-2.8){{2}};
\node at (-3,-1.8){{3}};
\node at (-1.8,-3){{1}};
\node at (-1,-4.2){{2}};
\node at (-0.8,-2.8){{3}};
\node at (-1,-1.8){{3}};
\node at (0.2,-3){{3}};
\node at (1,-4.2){{3}};
\node at (1.2,-2.8){{3}};
\node at (2.2,-3){{2}};
\node at (1,-1.8){{1}};
\node at (-0.2,-1){{3}};
\node at (1.2,-0.8){{2}};
\node at (1,0.2){{2}};
\node at (2.2,-1){{2}};
\node at (3,-2.2){{2}};
\node at (3.2,-0.8){{2}};
\node at (4.2,-1){{1}};
\node at (3,0.2){{3}};
\end{tikzpicture}}}
\begin{tikzpicture}[baseline=20mm]
\node at (0,0){$\mapsto$};
\end{tikzpicture}
\rotatebox{0}{\scalebox{0.6}{\begin{tikzpicture}[baseline=0mm]
\coordinate(1) at (0,0){}; 
\coordinate(2) at (2,0){}; 
\coordinate(3) at (4,0){}; 
\coordinate(4) at (-4,-2){};
\coordinate(5) at (-2,-2){}; 
\coordinate(6) at (0,-2){}; 
\coordinate(7) at (2,-2){};
\coordinate(8) at (4,-2){};
\coordinate(9) at (-4,-4){};
\coordinate(10) at (-2,-4){}; 
\coordinate(11) at (0,-4){}; 
\coordinate(12) at (2,-4){};
\coordinate(13) at (-4,-6){}; 
\coordinate(14) at (-2,-6){}; 
\draw(1) to (3);
\draw(4) to (8);
\draw(9) to (12);
\draw(13) to (14);
\draw(3) to (8);
\draw(2) to (12);
\draw(1) to (11);
\draw(5) to (14);
\draw(4) to (13);
\node at (-3,-6.2){{1}};
\node at (-4.2,-5){{3}};
\node at (-1.8,-5){{2}};
\node at (-3,-3.8){{2}};
\node at (-4.2,-3){{2}};
\node at (-3,-1.8){{3}};
\node at (-1.8,-3){{1}};
\node at (-1,-4.2){{2}};
\node at (-1,-1.8){{3}};
\node at (0.2,-3){{3}};
\node at (1,-4.2){{3}};
\node at (2.2,-3){{2}};
\node at (1,-1.8){{1}};
\node at (-0.2,-1){{3}};
\node at (1,0.2){{2}};
\node at (2.2,-1){{2}};
\node at (3,-2.2){{2}};
\node at (4.2,-1){{1}};
\node at (3,0.2){{3}};
\end{tikzpicture}}}
\end{figure}
\end{example}
By using the snake graph, we give cluster variables in the cluster pattern $CP_{(\xx,\Delta,2I_3)}$.

Let $\gamma$ be an arc in $\mathcal O_3$. We assume that $\gamma$ crosses $l_{i_1},\dots,l_{i_m}$ in order (except for its endpoints), where $i_j\in\{1,2,3\}$. We define $\mathrm{cross}(T_0,\gamma)$ as
\[\mathrm{cross}(T_0,\gamma)=\prod_{j=1}^m x_{i_j}.\]
Let $P$ be a perfect matching of $G_{T_0,\gamma}$. When $P$ consists edges labeled $i_1,\dots,i_k$, we define the \emph{weight} of $P$ as 
\[x(P)=\prod_{j=1}^k x_{i_j}.\]
\begin{remark}
The assumption that orders of the orbifold points of $\mathcal O_3$ are all 3 is used here. In general, when the order of an orbifold point is $p$, the variable associated with an edge in $G_{T_0,\gamma}$ that corresponds to the monogon $l_i$ surrounding this orbifold point in $\mathcal O_3$ is $2x_i\cos \pi/p$.
\end{remark}
The following theorem is a special case of \cite{bake}*{Theorem 1.1}.
\begin{theorem}\label{clustervariable-arc}\noindent
\begin{itemize}
    \item[(1)] For an arc $\gamma$ in $\mathcal{O}_3$, 
\[x_{\gamma}=\dfrac{1}{\mathrm{cross}(T_0,\gamma)}\sum_Px(P)\]
is a cluster variable of the cluster pattern $CP_{(\xx,\Delta,2I_3)}$, where the sum runs over all the perfect matching $P$. 

 \item[(2)]The correspondence $\varphi\colon \gamma\mapsto x_\gamma$ is a bijection from arcs in $\mathcal O_3$ to cluster variables in $CP_{(\xx,\Delta,2I_3)}$, and induces a bijection $\Phi$ from clusters in $CP_{(\xx,\Delta,2I_3)}$ and triangulations of $\mathcal O_3$.
 \item[(3)] The bijection $\Phi$ is compatible with mutations and flips, that is, for any triangulation $T$, we have
 \[\Phi(f_k(T))=\mu_k(\Phi(T)).\]
\end{itemize}
\end{theorem}

If we substitute $x_1=x_2=x_3=1$ in the formula of Theorem \ref{clustervariable-arc} (1), then we have the following corollary by Theorem \ref{clusterstructure}:

\begin{corollary}\label{cor:number-matching}\noindent
\begin{itemize}
    \item [(1)]
Let $\gamma$ be an arc in $\mathcal O_3$ and $G_{T_0,\gamma}$ the snake graph of $\gamma$. Then, the number of perfect matchings of $G_{T_0,\gamma}$ appears in the tree $\mathbb T$.
\item[(2)]
Let $T=(\gamma_1,\gamma_2,\gamma_3)$ be a triangulation of $\mathcal O_3$. Then, the triplet of numbers of perfect matchings of $G_{T_0,\gamma_1},G_{T_0,\gamma_2},G_{T_0,\gamma_3}$ is a positive integer solution to \eqref{Diophantine}. 
\item[(3)] In the correspondence of (2) from triangulations to integer solutions, the integer solution corresponding to $f_k(T)$ is the neighbor solution of that corresponding to $T$ in $\TT'_3$, connected by an edge labeled by $k$.
\end{itemize}
\end{corollary}
\begin{example}
Let $\gamma$ be one in Example \ref{gamma-ex}. Then, we have $13$ perfect matchings (see Figure \ref{figure10}). We have
\[x_{\gamma}=\dfrac{x_1^4 + 2 x_1^3 x_2 + x_1^3 x_3 + 3 x_1^2 x_2^2 + x_1^2 x_2 x_3 + x_1^2 x_3^2 + 2 x_1 x_2^3 + x_1 x_2^2 x_3 + x_2^4}{x_2 x_3^2}.\]
This is the second element of $\mu_2\mu_3(x_1,x_2,x_3)$ and ``13" appears in $\TT$.
\begin{figure}[ht]
    \centering
    \caption{Perfect matchings of $G_{T_0,\gamma}$}
    \label{figure10}
\scalebox{3}{\begin{tikzpicture}[baseline=0mm]
\coordinate(1) at (0,0){}; 
\coordinate(2) at (0.2,0){}; 
\coordinate(3) at (0.4,0){}; 
\coordinate(4) at (-0.4,-0.2){};
\coordinate(5) at (-0.2,-0.2){}; 
\coordinate(6) at (0,-0.2){}; 
\coordinate(7) at (0.2,-0.2){};
\coordinate(8) at (0.4,-0.2){};
\coordinate(9) at (-0.4,-0.4){};
\coordinate(10) at (-0.2,-0.4){}; 
\coordinate(11) at (0,-0.4){}; 
\coordinate(12) at (0.2,-0.4){};
\coordinate(13) at (-0.4,-0.6){}; 
\coordinate(14) at (-0.2,-0.6){}; 
\draw[very thick, red](1) to (2);
\draw(2) to (3);
\draw[very thick, red](4) to (5);
\draw(5) to (6);
\draw[very thick, red](6) to (7);
\draw(7) to (8);
\draw[very thick, red](9) to (10);
\draw(10) to (11);
\draw[very thick, red](11) to (12);
\draw[very thick, red](13) to (14);
\draw[very thick, red](3) to (8);
\draw(2) to (7);
\draw(7) to (12);
\draw(1) to (6);
\draw(6) to (11);
\draw(5) to (10);
\draw(10) to (14);
\draw(4) to (9);
\draw(9) to (13);
\end{tikzpicture}}\hspace{3mm}
\scalebox{3}{\begin{tikzpicture}[baseline=0mm]
\coordinate(1) at (0,0){}; 
\coordinate(2) at (0.2,0){}; 
\coordinate(3) at (0.4,0){}; 
\coordinate(4) at (-0.4,-0.2){};
\coordinate(5) at (-0.2,-0.2){}; 
\coordinate(6) at (0,-0.2){}; 
\coordinate(7) at (0.2,-0.2){};
\coordinate(8) at (0.4,-0.2){};
\coordinate(9) at (-0.4,-0.4){};
\coordinate(10) at (-0.2,-0.4){}; 
\coordinate(11) at (0,-0.4){}; 
\coordinate(12) at (0.2,-0.4){};
\coordinate(13) at (-0.4,-0.6){}; 
\coordinate(14) at (-0.2,-0.6){}; 
\draw(1) to (2);
\draw[very thick, red](2) to (3);
\draw[very thick, red](4) to (5);
\draw(5) to (6);
\draw(6) to (7);
\draw[very thick, red](7) to (8);
\draw[very thick, red](9) to (10);
\draw(10) to (11);
\draw[very thick, red](11) to (12);
\draw[very thick, red](13) to (14);
\draw(3) to (8);
\draw(2) to (7);
\draw(7) to (12);
\draw[very thick, red](1) to (6);
\draw(6) to (11);
\draw(5) to (10);
\draw(10) to (14);
\draw(4) to (9);
\draw(9) to (13);
\end{tikzpicture}}\hspace{3mm}
\scalebox{3}{\begin{tikzpicture}[baseline=0mm]
\coordinate(1) at (0,0){}; 
\coordinate(2) at (0.2,0){}; 
\coordinate(3) at (0.4,0){}; 
\coordinate(4) at (-0.4,-0.2){};
\coordinate(5) at (-0.2,-0.2){}; 
\coordinate(6) at (0,-0.2){}; 
\coordinate(7) at (0.2,-0.2){};
\coordinate(8) at (0.4,-0.2){};
\coordinate(9) at (-0.4,-0.4){};
\coordinate(10) at (-0.2,-0.4){}; 
\coordinate(11) at (0,-0.4){}; 
\coordinate(12) at (0.2,-0.4){};
\coordinate(13) at (-0.4,-0.6){}; 
\coordinate(14) at (-0.2,-0.6){}; 
\draw(1) to (2);
\draw(2) to (3);
\draw[very thick, red](4) to (5);
\draw(5) to (6);
\draw(6) to (7);
\draw(7) to (8);
\draw[very thick, red](9) to (10);
\draw(10) to (11);
\draw[very thick, red](11) to (12);
\draw[very thick, red](13) to (14);
\draw[very thick, red](3) to (8);
\draw[very thick, red](2) to (7);
\draw(7) to (12);
\draw[very thick, red](1) to (6);
\draw(6) to (11);
\draw(5) to (10);
\draw(10) to (14);
\draw(4) to (9);
\draw(9) to (13);
\end{tikzpicture}}\hspace{3mm}
\scalebox{3}{\begin{tikzpicture}[baseline=0mm]
\coordinate(1) at (0,0){}; 
\coordinate(2) at (0.2,0){}; 
\coordinate(3) at (0.4,0){}; 
\coordinate(4) at (-0.4,-0.2){};
\coordinate(5) at (-0.2,-0.2){}; 
\coordinate(6) at (0,-0.2){}; 
\coordinate(7) at (0.2,-0.2){};
\coordinate(8) at (0.4,-0.2){};
\coordinate(9) at (-0.4,-0.4){};
\coordinate(10) at (-0.2,-0.4){}; 
\coordinate(11) at (0,-0.4){}; 
\coordinate(12) at (0.2,-0.4){};
\coordinate(13) at (-0.4,-0.6){}; 
\coordinate(14) at (-0.2,-0.6){}; 
\draw[very thick, red](1) to (2);
\draw(2) to (3);
\draw[very thick, red](4) to (5);
\draw(5) to (6);
\draw(6) to (7);
\draw(7) to (8);
\draw[very thick, red](9) to (10);
\draw(10) to (11);
\draw(11) to (12);
\draw[very thick, red](13) to (14);
\draw[very thick, red](3) to (8);
\draw(2) to (7);
\draw[very thick, red](7) to (12);
\draw(1) to (6);
\draw[very thick, red](6) to (11);
\draw(5) to (10);
\draw(10) to (14);
\draw(4) to (9);
\draw(9) to (13);
\end{tikzpicture}}\hspace{3mm}
\scalebox{3}{\begin{tikzpicture}[baseline=0mm]
\coordinate(1) at (0,0){}; 
\coordinate(2) at (0.2,0){}; 
\coordinate(3) at (0.4,0){}; 
\coordinate(4) at (-0.4,-0.2){};
\coordinate(5) at (-0.2,-0.2){}; 
\coordinate(6) at (0,-0.2){}; 
\coordinate(7) at (0.2,-0.2){};
\coordinate(8) at (0.4,-0.2){};
\coordinate(9) at (-0.4,-0.4){};
\coordinate(10) at (-0.2,-0.4){}; 
\coordinate(11) at (0,-0.4){}; 
\coordinate(12) at (0.2,-0.4){};
\coordinate(13) at (-0.4,-0.6){}; 
\coordinate(14) at (-0.2,-0.6){}; 
\draw[very thick, red](1) to (2);
\draw(2) to (3);
\draw(4) to (5);
\draw(5) to (6);
\draw[very thick, red](6) to (7);
\draw(7) to (8);
\draw(9) to (10);
\draw(10) to (11);
\draw[very thick, red](11) to (12);
\draw[very thick, red](13) to (14);
\draw[very thick, red](3) to (8);
\draw(2) to (7);
\draw(7) to (12);
\draw(1) to (6);
\draw(6) to (11);
\draw[very thick, red](5) to (10);
\draw(10) to (14);
\draw[very thick, red](4) to (9);
\draw(9) to (13);
\end{tikzpicture}}\\
\scalebox{3}{\begin{tikzpicture}[baseline=0mm]
\coordinate(1) at (0,0){}; 
\coordinate(2) at (0.2,0){}; 
\coordinate(3) at (0.4,0){}; 
\coordinate(4) at (-0.4,-0.2){};
\coordinate(5) at (-0.2,-0.2){}; 
\coordinate(6) at (0,-0.2){}; 
\coordinate(7) at (0.2,-0.2){};
\coordinate(8) at (0.4,-0.2){};
\coordinate(9) at (-0.4,-0.4){};
\coordinate(10) at (-0.2,-0.4){}; 
\coordinate(11) at (0,-0.4){}; 
\coordinate(12) at (0.2,-0.4){};
\coordinate(13) at (-0.4,-0.6){}; 
\coordinate(14) at (-0.2,-0.6){}; 
\draw(1) to (2);
\draw[very thick, red](2) to (3);
\draw(4) to (5);
\draw(5) to (6);
\draw(6) to (7);
\draw[very thick, red](7) to (8);
\draw(9) to (10);
\draw(10) to (11);
\draw[very thick, red](11) to (12);
\draw[very thick, red](13) to (14);
\draw(3) to (8);
\draw(2) to (7);
\draw(7) to (12);
\draw[very thick, red](1) to (6);
\draw(6) to (11);
\draw[very thick, red](5) to (10);
\draw(10) to (14);
\draw[very thick, red](4) to (9);
\draw(9) to (13);
\end{tikzpicture}}\hspace{3mm}
\scalebox{3}{\begin{tikzpicture}[baseline=0mm]
\coordinate(1) at (0,0){}; 
\coordinate(2) at (0.2,0){}; 
\coordinate(3) at (0.4,0){}; 
\coordinate(4) at (-0.4,-0.2){};
\coordinate(5) at (-0.2,-0.2){}; 
\coordinate(6) at (0,-0.2){}; 
\coordinate(7) at (0.2,-0.2){};
\coordinate(8) at (0.4,-0.2){};
\coordinate(9) at (-0.4,-0.4){};
\coordinate(10) at (-0.2,-0.4){}; 
\coordinate(11) at (0,-0.4){}; 
\coordinate(12) at (0.2,-0.4){};
\coordinate(13) at (-0.4,-0.6){}; 
\coordinate(14) at (-0.2,-0.6){}; 
\draw(1) to (2);
\draw(2) to (3);
\draw(4) to (5);
\draw(5) to (6);
\draw(6) to (7);
\draw(7) to (8);
\draw(9) to (10);
\draw(10) to (11);
\draw[very thick, red](11) to (12);
\draw[very thick, red](13) to (14);
\draw[very thick, red](3) to (8);
\draw[very thick, red](2) to (7);
\draw(7) to (12);
\draw[very thick, red](1) to (6);
\draw(6) to (11);
\draw[very thick, red](5) to (10);
\draw(10) to (14);
\draw[very thick, red](4) to (9);
\draw(9) to (13);
\end{tikzpicture}}\hspace{3mm}
\scalebox{3}{\begin{tikzpicture}[baseline=0mm]
\coordinate(1) at (0,0){}; 
\coordinate(2) at (0.2,0){}; 
\coordinate(3) at (0.4,0){}; 
\coordinate(4) at (-0.4,-0.2){};
\coordinate(5) at (-0.2,-0.2){}; 
\coordinate(6) at (0,-0.2){}; 
\coordinate(7) at (0.2,-0.2){};
\coordinate(8) at (0.4,-0.2){};
\coordinate(9) at (-0.4,-0.4){};
\coordinate(10) at (-0.2,-0.4){}; 
\coordinate(11) at (0,-0.4){}; 
\coordinate(12) at (0.2,-0.4){};
\coordinate(13) at (-0.4,-0.6){}; 
\coordinate(14) at (-0.2,-0.6){}; 
\draw[very thick, red](1) to (2);
\draw(2) to (3);
\draw(4) to (5);
\draw(5) to (6);
\draw(6) to (7);
\draw(7) to (8);
\draw(9) to (10);
\draw(10) to (11);
\draw(11) to (12);
\draw[very thick, red](13) to (14);
\draw[very thick, red](3) to (8);
\draw(2) to (7);
\draw[very thick, red](7) to (12);
\draw(1) to (6);
\draw[very thick, red](6) to (11);
\draw[very thick, red](5) to (10);
\draw(10) to (14);
\draw[very thick, red](4) to (9);
\draw(9) to (13);
\end{tikzpicture}}\hspace{3mm}
\scalebox{3}{\begin{tikzpicture}[baseline=0mm]
\coordinate(1) at (0,0){}; 
\coordinate(2) at (0.2,0){}; 
\coordinate(3) at (0.4,0){}; 
\coordinate(4) at (-0.4,-0.2){};
\coordinate(5) at (-0.2,-0.2){}; 
\coordinate(6) at (0,-0.2){}; 
\coordinate(7) at (0.2,-0.2){};
\coordinate(8) at (0.4,-0.2){};
\coordinate(9) at (-0.4,-0.4){};
\coordinate(10) at (-0.2,-0.4){}; 
\coordinate(11) at (0,-0.4){}; 
\coordinate(12) at (0.2,-0.4){};
\coordinate(13) at (-0.4,-0.6){}; 
\coordinate(14) at (-0.2,-0.6){}; 
\draw[very thick, red](1) to (2);
\draw(2) to (3);
\draw[very thick, red](4) to (5);
\draw(5) to (6);
\draw[very thick, red](6) to (7);
\draw(7) to (8);
\draw(9) to (10);
\draw(10) to (11);
\draw[very thick, red](11) to (12);
\draw(13) to (14);
\draw[very thick, red](3) to (8);
\draw(2) to (7);
\draw(7) to (12);
\draw(1) to (6);
\draw(6) to (11);
\draw(5) to (10);
\draw[very thick, red](10) to (14);
\draw(4) to (9);
\draw[very thick, red](9) to (13);
\end{tikzpicture}}\hspace{3mm}
\scalebox{3}{\begin{tikzpicture}[baseline=0mm]
\coordinate(1) at (0,0){}; 
\coordinate(2) at (0.2,0){}; 
\coordinate(3) at (0.4,0){}; 
\coordinate(4) at (-0.4,-0.2){};
\coordinate(5) at (-0.2,-0.2){}; 
\coordinate(6) at (0,-0.2){}; 
\coordinate(7) at (0.2,-0.2){};
\coordinate(8) at (0.4,-0.2){};
\coordinate(9) at (-0.4,-0.4){};
\coordinate(10) at (-0.2,-0.4){}; 
\coordinate(11) at (0,-0.4){}; 
\coordinate(12) at (0.2,-0.4){};
\coordinate(13) at (-0.4,-0.6){}; 
\coordinate(14) at (-0.2,-0.6){}; 
\draw(1) to (2);
\draw[very thick, red](2) to (3);
\draw[very thick, red](4) to (5);
\draw(5) to (6);
\draw(6) to (7);
\draw[very thick, red](7) to (8);
\draw(9) to (10);
\draw(10) to (11);
\draw[very thick, red](11) to (12);
\draw(13) to (14);
\draw(3) to (8);
\draw(2) to (7);
\draw(7) to (12);
\draw[very thick, red](1) to (6);
\draw(6) to (11);
\draw(5) to (10);
\draw[very thick, red](10) to (14);
\draw(4) to (9);
\draw[very thick, red](9) to (13);
\end{tikzpicture}}\\
\scalebox{3}{\begin{tikzpicture}[baseline=0mm]
\coordinate(1) at (0,0){}; 
\coordinate(2) at (0.2,0){}; 
\coordinate(3) at (0.4,0){}; 
\coordinate(4) at (-0.4,-0.2){};
\coordinate(5) at (-0.2,-0.2){}; 
\coordinate(6) at (0,-0.2){}; 
\coordinate(7) at (0.2,-0.2){};
\coordinate(8) at (0.4,-0.2){};
\coordinate(9) at (-0.4,-0.4){};
\coordinate(10) at (-0.2,-0.4){}; 
\coordinate(11) at (0,-0.4){}; 
\coordinate(12) at (0.2,-0.4){};
\coordinate(13) at (-0.4,-0.6){}; 
\coordinate(14) at (-0.2,-0.6){}; 
\draw(1) to (2);
\draw(2) to (3);
\draw[very thick, red](4) to (5);
\draw(5) to (6);
\draw(6) to (7);
\draw(7) to (8);
\draw(9) to (10);
\draw(10) to (11);
\draw[very thick, red](11) to (12);
\draw(13) to (14);
\draw[very thick, red](3) to (8);
\draw[very thick, red](2) to (7);
\draw(7) to (12);
\draw[very thick, red](1) to (6);
\draw(6) to (11);
\draw(5) to (10);
\draw[very thick, red](10) to (14);
\draw(4) to (9);
\draw[very thick, red](9) to (13);
\end{tikzpicture}}\hspace{3mm}
\scalebox{3}{\begin{tikzpicture}[baseline=0mm]
\coordinate(1) at (0,0){}; 
\coordinate(2) at (0.2,0){}; 
\coordinate(3) at (0.4,0){}; 
\coordinate(4) at (-0.4,-0.2){};
\coordinate(5) at (-0.2,-0.2){}; 
\coordinate(6) at (0,-0.2){}; 
\coordinate(7) at (0.2,-0.2){};
\coordinate(8) at (0.4,-0.2){};
\coordinate(9) at (-0.4,-0.4){};
\coordinate(10) at (-0.2,-0.4){}; 
\coordinate(11) at (0,-0.4){}; 
\coordinate(12) at (0.2,-0.4){};
\coordinate(13) at (-0.4,-0.6){}; 
\coordinate(14) at (-0.2,-0.6){}; 
\draw[very thick, red](1) to (2);
\draw(2) to (3);
\draw[very thick, red](4) to (5);
\draw(5) to (6);
\draw(6) to (7);
\draw(7) to (8);
\draw(9) to (10);
\draw(10) to (11);
\draw(11) to (12);
\draw(13) to (14);
\draw[very thick, red](3) to (8);
\draw(2) to (7);
\draw[very thick, red](7) to (12);
\draw(1) to (6);
\draw[very thick, red](6) to (11);
\draw(5) to (10);
\draw[very thick, red](10) to (14);
\draw(4) to (9);
\draw[very thick, red](9) to (13);
\end{tikzpicture}}\hspace{3mm}
\scalebox{3}{\begin{tikzpicture}[baseline=0mm]
\coordinate(1) at (0,0){}; 
\coordinate(2) at (0.2,0){}; 
\coordinate(3) at (0.4,0){}; 
\coordinate(4) at (-0.4,-0.2){};
\coordinate(5) at (-0.2,-0.2){}; 
\coordinate(6) at (0,-0.2){}; 
\coordinate(7) at (0.2,-0.2){};
\coordinate(8) at (0.4,-0.2){};
\coordinate(9) at (-0.4,-0.4){};
\coordinate(10) at (-0.2,-0.4){}; 
\coordinate(11) at (0,-0.4){}; 
\coordinate(12) at (0.2,-0.4){};
\coordinate(13) at (-0.4,-0.6){}; 
\coordinate(14) at (-0.2,-0.6){}; 
\draw[very thick, red](1) to (2);
\draw(2) to (3);
\draw(4) to (5);
\draw[very thick, red](5) to (6);
\draw(6) to (7);
\draw(7) to (8);
\draw(9) to (10);
\draw[very thick, red](10) to (11);
\draw(11) to (12);
\draw[very thick, red](13) to (14);
\draw[very thick, red](3) to (8);
\draw(2) to (7);
\draw[very thick, red](7) to (12);
\draw(1) to (6);
\draw(6) to (11);
\draw(5) to (10);
\draw(10) to (14);
\draw[very thick, red](4) to (9);
\draw(9) to (13);
\end{tikzpicture}}
\end{figure}
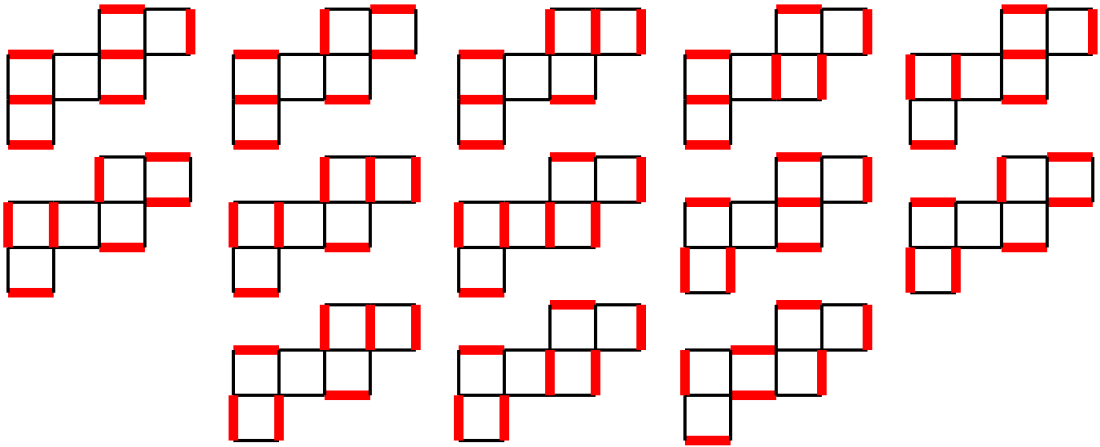
\end{example}
\begin{remark}
In Markov case, an analogue of Corollary \ref{cor:number-matching} was given by Propp in \cite{propp}.
\end{remark}
\subsection{Realization of integer solutions by the numerator of continued fraction}
In this subsection, we introduce another realization of positive integer solutions to \eqref{Diophantine} by using snake graph. 

Fix the orientation of the snake graph $G_{T_0,\gamma}$ so that the point corresponding to the start point of the $\gamma$ is in the southwest and the point corresponding to the terminal point is in the northeast. We will construct continued fraction from the snake graph $G_{T_0,\gamma}$. First, we attach a sign $+$ or $-$ to every edge of $G_{T_0,\gamma}$ using the following conditions: for every tile in $G_{T_0,\gamma}$, 
\begin{itemize}
    \item  the north and the west edge have the same sign, 
    \item the south and the east edge have the same sign,
    \item the sign on the north edge is opposite to the sign on the south edge. 
\end{itemize}
Such a way of attaching signs is uniquely determined by determining the sign of any one place, thus there are two ways of attaching a sign that are dual to each other. Either one will be fine for the discussion that follows. 
Second, we consider the curve $\tilde{\gamma}$ connecting the southwest point to the northeast point in $G_{T_0,\gamma}$, passing through the interior of the $G_{T_0,\gamma}$, such that the number of intersections with the edges of $G_{T_0,\gamma}$ is minimized. Third, we take the signs associated with the edges of $G_{T_0,\gamma}$ through which $\tilde{\gamma}$ passes and arrange them in order (the first sign is one associated the south edge, and the last sign is one associated the east edge). Finally, for the obtained sequence of signs, if the numbers followed by the same sign are $ a_1, ..., a_n $ in order, we give a continued fraction \[[a_1, ..., a_n]=a_1+\dfrac{1}{a_2+\dfrac{1}{a_3+\dfrac{1}{\ddots\dfrac{1}{a_{n-1}+\dfrac{1}{a_{n}}}}}}.\]
We denote by $\mathrm{Fr}(\gamma)$ the above continued fraction. The method of constructing continued fractions from snake graphs was given by \c{C}anak\c{c}{\i} and Schiffler in \cite{cs18}. See it for details. 

\begin{example}\label{example-contifrac}
Let $\gamma$ be one in Example \ref{gamma-ex}. Then signs of $G_{T_0,\gamma}$ is given as in the Figure 11 (red signs are the signs needed to construct a continued fraction). Then, we have
\[\mathrm{Fr}(\gamma)=[1,2,3,1]=1+\dfrac{1}{2+\dfrac{1}{3+\dfrac{1}{1}}}=\dfrac{13}{9}.\]
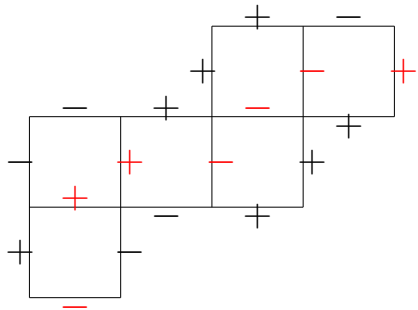
\begin{figure}[ht]
    \centering
    \caption{Sign attached each edge}
    \label{figure 11}
\rotatebox{0}{\scalebox{0.6}{\begin{tikzpicture}[baseline=0mm]
\coordinate(1) at (0,0){}; 
\coordinate(2) at (2,0){}; 
\coordinate(3) at (4,0){}; 
\coordinate(4) at (-4,-2){};
\coordinate(5) at (-2,-2){}; 
\coordinate(6) at (0,-2){}; 
\coordinate(7) at (2,-2){};
\coordinate(8) at (4,-2){};
\coordinate(9) at (-4,-4){};
\coordinate(10) at (-2,-4){}; 
\coordinate(11) at (0,-4){}; 
\coordinate(12) at (2,-4){};
\coordinate(13) at (-4,-6){}; 
\coordinate(14) at (-2,-6){}; 
\draw(1) to (3);
\draw(4) to (8);
\draw(9) to (12);
\draw(13) to (14);
\draw(3) to (8);
\draw(2) to (12);
\draw(1) to (11);
\draw(5) to (14);
\draw(4) to (13);
\node at (-3,-6.2){{\Huge \textcolor{red}{$-$}}};
\node at (-4.2,-5){{\Huge $+$}};
\node at (-1.8,-5){{\Huge $-$}};
\node at (-3,-3.8){{\Huge \textcolor{red}{$+$}}};
\node at (-4.2,-3){{\Huge $-$}};
\node at (-3,-1.8){{\Huge $-$}};
\node at (-1.8,-3){{\Huge \textcolor{red}{$+$}}};
\node at (-1,-4.2){{\Huge $-$}};
\node at (-1,-1.8){{\Huge $+$}};
\node at (0.2,-3){{\Huge \textcolor{red}{$-$}}};
\node at (1,-4.2){{\Huge $+$}};
\node at (2.2,-3){{\Huge $+$}};
\node at (1,-1.8){{\Huge \textcolor{red}{$-$}}};
\node at (-0.2,-1){{\Huge $+$}};
\node at (1,0.2){{\Huge $+$}};
\node at (2.2,-1){{\Huge \textcolor{red}{$-$}}};
\node at (3,-2.2){{\Huge $+$}};
\node at (4.2,-1){{\Huge \textcolor{red}{$+$}}};
\node at (3,0.2){{\Huge $-$}};
\end{tikzpicture}}}
\end{figure}
\end{example}

The following theorem is a special case of \cite{cs18}*{Theorem 3.4}:
\begin{theorem}
Let $\gamma$ be an arc in $\mathcal O_3$. Then, the number of perfect matchings of $G_{T_0,\gamma}$ equals to the numerator of $\mathrm{Fr}(\gamma)$. In particular, the numerator of $\mathrm{Fr}(\gamma)$ appears in $\mathbb T$.
\end{theorem}
In Example \ref{example-contifrac}, the numerator ``13" of $\mathrm{Fr}(\gamma)$ appears in $\TT$.

\bibliography{myrefs}
\end{document}